\documentclass[numbers,webpdf]{ima-authoring-template}%

\graphicspath{{Fig/}}

\theoremstyle{thmstyletwo}%
\newtheorem{theorem}{Theorem}
\newtheorem{proposition}[theorem]{Proposition}%
\newtheorem{lemma}[theorem]{Lemma}

\newtheorem{example}{Example}%
\newtheorem{remark}{Remark}%

\numberwithin{equation}{section}

\usepackage{subfig}
\usepackage{mathdots}
\newcommand{\norm}[1]{\left\|#1\right\|}
\newcommand{\inner}[1]{\left\langle #1\right\rangle}
\newcommand{\E}{\mathcal{E}}
\begin{document}

\DOI{DOI HERE}
\copyrightyear{2021}
\vol{00}
\pubyear{2021}
\access{Advance Access Publication Date: Day Month Year}
\appnotes{Paper}
\copyrightstatement{Published by Oxford University Press on behalf of the Institute of Mathematics and its Applications. All rights reserved.}
\firstpage{1}


\title[Algorithms for Monotone Inclusion Problem with Sums of H\"{o}lder Operators]{Convergence Rates of  Tseng’s Splitting Method and Its Acceleration Schemes for Monotone Inclusion Problem with a Sum of H\"{o}lder Continuous Operators}

\author{Yi Zhang*
\address{\orgdiv{Department of Applied Mathematics}, \orgname{The Hong Kong Polytechnic University}, \orgaddress{\street{Hung Hom}, \postcode{Kowloon}, \state{Hong Kong}, \country{China}}}}

\authormark{Yi Zhang}

\corresp[*]{Corresponding author: \href{email:yi8zhang@polyu.edu.hk}{yi8zhang@polyu.edu.hk}}

\received{20th}{6}{2026}
\revised{Date}{0}{Year}
\accepted{Date}{0}{Year}


\abstract{The monotone inclusion problem is fundamental in applied mathematics and is closely related to a wide range of practical applications. However, existing solution methods typically require the underlying operator to be Lipschitz continuous. Recently, H\"{o}lder continuity, a weaker condition than Lipschitz continuity, has proven useful in characterizing certain real-world problems. To bridge this gap, we investigate the convergence rates of the Tseng’s splitting method (a fundamental algorithm for monotone inclusion problem) and its two accelerated variants, the composite extra anchored gradient method and the symplectic composite extra gradient method, under the H\"{o}lder continuity assumption. Our numerical experiments demonstrate that the numerical performance of these algorithms aligns with their respective theoretical convergence rates.}


\maketitle

\section{Introduction}

Let $\mathcal{H}$ be a real Hilbert space equipped with the inner product $\inner{\cdot,\cdot}$ and the corresponding norm $\norm{\cdot}$. We denote  $2^\mathcal{H}$ as the collection of all subsets of $\mathcal{H}$. Throughout this paper, we consider the monotone inclusion problem
\begin{equation}
	0\in T(x):=F(x)+G(x),
	\label{eq:main}
\end{equation}
where $T:\mathcal{H}\to 2^{\mathcal{H}}$ is a monotone operator, $F:\mathcal{H}\to\mathcal{H}$ is a continuous operator, and $G:\mathcal{H}\to 2^{\mathcal{H}}$ is a maximally monotone operator. Moreover, $F=\sum_{i=1}^N F_i$, where each $F_i:\mathcal{H}\to\mathcal{H}$ is H\"{o}lder continuous with H\"{o}lder exponent $\alpha_i\in(0,1]$ and H\"{o}lder constant $L_i$, i.e.,
\begin{equation}
	\norm{F_i(x)-F_i(y)}\leq L_i\norm{x-y}^{\alpha_i},\quad\forall x, y\in\mathcal{H}.
	\label{eq:holder}
\end{equation} 
If $\alpha_i=1$, then \eqref{eq:holder} implies that $F_i$ is $L_i$-Lipschitz continuous. We denote the solution set by $\mathrm{SOL}:=\{x\in\mathcal{H}\mid 0\in T(x)\}$ and assume that $\mathrm{SOL}\neq\emptyset$.

The monotone inclusion problem \eqref{eq:main} has a strong connection with several important classes of applied mathematical problems, such as optimization problems, minimax problems, duality theory, and variational inequality (VI) problems. These problem classes play an important role in numerous practical applications; see, e.g., \cite{TOYASAKI2014340,bental09,kinderlehrer2000,UI2016139,taylor84}. Since many algorithms for solving \eqref{eq:main} are rooted in algorithms for solving VI problems, we discuss the relationship between the VI problem and the monotone inclusion problem \eqref{eq:main} in detail.
\begin{example}[VI Problem]
	The Stampacchia variational inequality problem consists in finding a point $x^*\in\mathcal{C}\subset\mathcal{H}$ such that
	\begin{equation}
		\inner{F(x^*),x-x^*}\geq 0,\quad\forall x\in\mathcal{C}.
		\label{eq:vi}
	\end{equation}
	If $\mathcal{C}$ is a closed convex set, \eqref{eq:vi} can be reformulated as
	\[
	0\in F(x^*)+N_\mathcal{C}(x^*),
	\]
	where $N_\mathcal{C}$ is the normal cone of $\mathcal{C}$ defined by
	\[
	N_\mathcal{C}(x^*):=\begin{cases}
		\{u\mid \inner{u,x-x^*}\leq 0,\ \forall x\in\mathcal{C}\},& \text{if } x^*\in\mathcal{C};\\
		\emptyset,&\text{if } x^*\notin\mathcal{C}.
	\end{cases}
	\]
\end{example}

Due to the wide range of applications of \eqref{eq:main} in practical problems, the study of algorithms for solving \eqref{eq:main} has attracted increasing interest. First, we give a concise overview of existing algorithms under the assumption that $F$ is Lipschitz continuous. The extra-gradient (EG) method, proposed independently in \cite{antipin1976,korpelevich76}, and Popov's method (also known as the OGDA method \cite{daskalakis2018b,daskalakis2018a}) , proposed in \cite{popov80}, are two classic algorithms for solving the VI problem \eqref{eq:vi} where $F$ is monotone on $\mathcal{C}$. The Tseng's splitting method \cite{tseng00} was developed to solve \eqref{eq:main} when $G$ is a maximally monotone operator. The mirror-prox method introduced in \cite{Nemirovski2004} studies the VI problem under the assumption that $F$ is monotone and Lipschitz continuous with respect to a non-Euclidean norm on $\mathcal{C}$. The EG+ method \cite{diakonikolas21} addresses the equation \(0 = F(x)\) under the weak Minty variational inequality (MVI) assumption. Further extensions include the EG method with line search \cite{cai2021}, which can solve the VI problem when $F$ is continuous and pseudomonotone with a convergence guarantee, and the adaptive EG+ method \cite{fan23,pethick22}, which incorporates adaptive step-sizes when $F$ satisfies the weak MVI assumption. As stochastic optimization continues to gain prominence, research on stochastic EG methods has also grown significantly; see, e.g., \cite{gorbunov22,iusem17,kannan19,mishchenko20}. For a comprehensive overview of the EG method and its variants, we refer the reader to the recent survey \cite{tran23}.

Since Nesterov's accelerated gradient method \cite{nesterov83} achieves a higher-order convergence rate than the vanilla gradient method while maintaining the same per-iteration computational complexity, the study of accelerated first-order methods for solving \eqref{eq:main} has also gained significant attention. Halpern's iteration, originally proposed for finding fixed points of nonexpansive operators (1-Lipschitz continuous operators) \cite{halpern67}, has been widely studied \cite{lieder2021,qi21} and has found applications as an acceleration technique \cite{chen2025,lu2024,sun2025a}. Consequently, several accelerated methods have been derived from Halpern's iteration. Yoon and Ryu \cite{yoon21} combined Halpern's iteration with the EG method to propose the extra anchored gradient (EAG) method for solving $0=F(x)$ where $F$ is monotone. In \cite{lee21}, Lee and Kim further combined the EAG method with the EG+ method to obtain the fast EG (FEG) method, which solves $0=F(x)$ under a non-monotonicity assumption. Yang et al. \cite{yang2024} later extended the EAG and FEG methods to solve \eqref{eq:main} without assuming $G=0$. Other acceleration schemes have also been developed. For instance, the symplectic EG method \cite{yuan2025}, derived from the symplectic acceleration technique \cite{yuan2026}, and the moving anchored EG method improve upon Halpern's iteration to achieve faster convergence rates, with their extensions presented in \cite{trandinh2025}. Tran et al. \cite{tran2024} used Nesterov's accelerated gradient method to derive an accelerated forward-backward-forward splitting method. In \cite{bot2023fast,bot2023project}, Bo\c{t} et al. proposed accelerated variants of OGDA, called fOGDA-VI, by discretizing an ODE. Qu et al. \cite{qu2026extra} combined the EG method with Anderson(1) acceleration \cite{anderson1965} to obtain the EG-Anderson(1) method.

Several works have studied algorithms for solving \eqref{eq:main} in the case where $F$ is H\"{o}lder continuous. Nemirovski \cite{Nemirovski2004} originally proposed the mirror-prox method for the VI problem when $F$ is monotone on a compact convex set $\mathcal{C}$, and showed that the convergence rate of the Minty gap function is $O\bigl(k^{-\frac{1+\alpha}{2}}\bigr)$. Dang and Lan \cite{dang2015} further extended the mirror-prox method to solve the VI problem when $F$ is pseudomonotone and $\mathcal{C}$ is closed and compact, establishing that
\[
\min_{0\leq k\leq K}\norm{\frac{1}{s_{k+1}}\bigl(x_k-P_\mathcal{C}(x_k-s_{k+1}F(x_k))\bigr)}^2 \leq O\left(\frac{1}{\alpha^{\frac{\alpha+2}{\alpha}}K^\alpha}\right).
\]
Stonyakin et al. \cite{stonyakin2022} studied a generalization of the mirror-prox method when the H\"{o}lder exponent or H\"{o}lder constant is unknown, as well as an inexact oracle variant of the method. \cite{chakraborty2025} combined the Popov method with the mirror-prox method to obtain the Popov mirror-prox method, and studied its convergence behavior under both deterministic and stochastic settings.
\subsection{Our Contributions}

To the best of our knowledge, there is no work that studies algorithms for solving \eqref{eq:main} where $G$ is a maximally monotone operator or $F$ is a sum of H\"{o}lder continuous operators. The reason why we study the case that $F$ is the sum of H\"{o}lder continuous operator is that the sum of H\"{o}lder continuous operator is not a H\"{o}lder continuous operator, as illustrated by the following example.

\begin{example} \cite[Example 1]{chen2025}
	\label{exp:univ}
	Consider $F:\mathbb{R}\to\mathbb{R}$ defined by
	\[
	F = F_1 + F_2,\quad F_1(x)=x,\quad F_2(x)=\operatorname{sign}(x)\,x^{\frac{1}{2}}.
	\]
	In \cite{chen2025}, Chen et al. show that $F$ is locally, but not globally, Hölder continuous.
	
	This simple example demonstrates that in problem \eqref{eq:main}, the operator $F$, expressed as a sum of operators $F_n$ each of which is Hölder continuous, may itself fail to be Hölder continuous.
\end{example} 

In this paper, we aim to develop algorithms for solving \eqref{eq:main} with convergence rate guarantees. We first study the convergence rate of the Tseng's splitting method introduced by Tseng \cite{tseng00} under that $F=\sum_{i=1}^{N}F_i$ and $F_i$ satisfies \eqref{eq:holder}. We prove that if
\[
s_{k+1}=\frac{ds}{(k+1)^{\frac{1-\alpha}{2}}},
\]
where $\alpha=\min_{1\leq i\leq N}\alpha_i$, $d\in(0,1)$, and $0<s\leq\min_{1\leq i\leq N}\frac{1}{NL_i\sqrt{\alpha_i}}$, then the convergence rate of Tseng's splitting method satisfies
\[
\min_{0\leq k\leq K}\mathrm{dist}\bigl(0,T(y_{k+1})\bigr)^2 \leq O\!\left(\!\left(\frac{(2+2d^2)\alpha}{1-d^2}+\sum_{i=1}^N\frac{(1-\alpha_i)\alpha\ln(K)}{N(1-d^2)\alpha_i}\right)K^{-\alpha}\right).
\]

Next, we extend the convergence rates of two accelerated Tseng's splitting methods: the composite FEG (com-FEG) method \cite{yang2024} and the symplectic composite extra-gradient (SCEG) method \cite{yuan2025}, under the assumption that $F=\sum_{i=1}^NF_i$ and $F_i$ satisfies \eqref{eq:holder} for all $i=1, \cdots, N$. We prove that if $\alpha_i\in(\frac{1}{3},1]$ for all $i=1,\dots,N$, and
\[
s_{k+1}=\frac{s}{(k+1)^{\frac{3(1-\alpha)}{2}}},
\]
where $\alpha=\min_{1\leq i\leq N}\alpha_i$ and $0<s\leq\min_{1\leq i\leq N}\frac{1}{(N^2L_i^2\alpha_i)^{\frac{1}{2\alpha_i}}}$, then the convergence rate of the proposed algorithm is
\[
\mathrm{dist}\bigl(0,T(x_K)\bigr)^2 \leq O\!\left(\!\left(\frac{1}{(3\alpha-1)^2s^2}+\sum_{i=1}^N\frac{(1-\alpha_i)\ln(K)}{N\alpha_i}\right)K^{-(3\alpha-1)}\right).
\]
\begin{remark}
	If $\alpha>\frac{1}{2}$, the convergence rate of the com-FEG method and the SCEG method is faster than the convergence rate of the Tseng's splitting method. If $\alpha<\frac{1}{2}$, the convergence rate of the Tseng's splitting method is faster than the com-FEG method and the SCEG method. Our numerical experiments also support this theoretical result.
\end{remark}
\subsection{Preliminaries}
A set-valued operator $T:\mathcal{H}\to 2^{\mathcal{H}}$ is said to be monotone if\[
\inner{u-v,x-y}\geq 0,\quad\forall x, y\in\mathcal{H}, u\in T(x), v\in T(y).
\]
An operator $G:\mathcal{H}\to 2^{\mathcal{H}}$ is a maximally monotone operator if $G$ is a monotone operator and the graph of $G$, defined as $\mathrm{graph}(G):=\{(x,u)|u\in G(x)\}$, is not a proper subset of the graph of another monotone operator.

The proximal point operator of a maximally monotone operator $G$ is defined as $J_{sG}:=(I+sG)^{-1}$, where $I:\mathcal{H}\to\mathcal{H}$ is the identity map and $s>0$. By invoking the monotonicity of $G$ and the Minty’s surjectivity theorem \cite{minty62}, $J_{sG}$ is a single-valued mapping defined on $\mathcal{H}$. 

Given a subset $\mathcal{S}$ and a point $x$, the distance between $x$ and $\mathcal{S}$ is defined as \[
\mathrm{dist}(x,\mathcal{S})=\inf_{y\in\mathcal{S}}\norm{x-y}.
\]
\section{Tseng's Splitting Method}
First, we study the convergence rate of the Tseng's splitting method under the assumption that $F=\sum_{i=1}^NF_i$, where $F_i$ satisfies \eqref{eq:holder}. The Tseng's spliiting method is a first-order method to solve \eqref{eq:main}, whose iteration formula is given by:	start with $x_0$,
\begin{align}
	y_{k+1} &= J_{s_{k+1}G}(x_k-s_{k+1}F(x_k)), \label{eq:tseng1}\\
	x_{k+1} &= y_{k+1}-s_{k+1}(F(y_{k+1})-F(x_k)),\label{eq:tseng2}
\end{align}
where $s_{k+1}>0$. 
\begin{remark}
	It has been proven that if $F$ is $L$-Lipschitz continuous and $s_{k+1}\equiv s\in(0,\frac{1}{L})$, the convergence rate of the Tseng's splitting method is\[
	\min_{0\leq k\leq K}\mathrm{dist}(0,T(y_{k+1}))^2\leq O(K^{-1}).
	\]
\end{remark}
\begin{theorem}
	Let $\{x_k\}$ and $\{y_{k+1}\}$ be the sequences generated by \eqref{eq:tseng1} and \eqref{eq:tseng2}. If $F=\sum_{i=1}^NF_i$, $F_i$ satisfies \eqref{eq:holder} for all $i=1, \cdots, N$, $G$ is a maximally monotone operator, $T$ is monotone and 
\begin{equation}
		s_{k+1}=\frac{ds}{(k+1)^{\frac{1-\alpha}{2}}},
		\label{eq:stepsizeT}
\end{equation}
	where $\alpha=\min_{1\leq i\leq N}\alpha_i$, $d\in(0,1)$, $0<s\leq\min_{1\leq i\leq N}\left(\frac{1}{NL_i\sqrt{\alpha_i}}\right)$, then we have
	\[
	\min_{0\leq k\leq K}\mathrm{dist}(0,T(y_{k+1}))^2\leq O\left(\left(\frac{(2+2d^2)\alpha}{1-d^2}+\sum_{i=1}^N\frac{(1-\alpha_i)\alpha d^2\ln(K)}{N(1-d^2)\alpha_i}\right)K^{-\alpha}\right).
	\]
\end{theorem}
\begin{proof}
	We begin by setting $\tilde{g}_{k+1}=s_{k+1}^{-1}(x_k-y_{k+1})-F(x_k)$. 
	From the definition of $J_{s_{k+1}G}$, it immediately follows that $\tilde{g}_{k+1}\in G(y_{k+1})$. 
	For an arbitrary point $x^*\in \mathrm{SOL}$, the following chain of equalities holds:
	\begin{align*}
		&\ \frac{1}{2}\norm{x_{k+1}-x^*}^2\\
		=&\ \frac{1}{2}\norm{y_{k+1}-x^*-s_{k+1}(F(y_{k+1})-F(x_k))}^2\\
		=&\ \frac{1}{2}\norm{y_{k+1}-x^*}^2-s_{k+1}\inner{F(y_{k+1})-F(x_k),y_{k+1}-x^*}+\frac{s_{k+1}^2}{2}\norm{F(y_{k+1})-F(x_k)}^2\\
		=&\ \frac{1}{2}\norm{x_k-x^*}^2+\frac{1}{2}\norm{y_{k+1}-x_k}^2+\inner{y_{k+1}-x_k,x_k-x^*}\\
		&\ -s_{k+1}\inner{F(y_{k+1})-F(x_k),y_{k+1}-x^*}+\frac{s_{k+1}^2}{2}\norm{F(y_{k+1})-F(x_k)}^2\\
		=&\ \frac{1}{2}\norm{x_k-x^*}-\frac{1}{2}\norm{y_{k+1}-x_k}^2+\frac{s_{k+1}^2}{2}\norm{F(y_{k+1})-F(x_k)}^2-s_{k+1}\inner{F(y_{k+1})+\tilde{g}_{k+1},y_{k+1}-x^*}.
	\end{align*}
	Applying the assumption on $F$ yields
	\begin{equation}
		\begin{split}
			&\ \frac{s_{k+1}^2}{2}\norm{F(y_{k+1})-F(x_k)}^2\\
			=&\ \frac{s_{k+1}^2}{2}\norm{\sum_{i=1}^N\left(F_i(y_{k+1})-F_i(x_k)\right)}^2\\
			\leq&\ \frac{Ns_{k+1}^2}{2}\sum_{i=1}^N\norm{F_i(y_{k+1})-F_i(x_k)}^2\\
			\leq&\ \frac{Ns_{k+1}^2}{2}\sum_{i=1}^NL_i^2\norm{y_{k+1}-x_k}^{2\alpha_i}\\
			=&\ \sum_{i=1}^N\left(\frac{NL_i^2s_{k+1}^2}{2}\norm{y_{k+1}-x_k}^{2\alpha_i}-\frac{d^2}{2N}\norm{y_{k+1}-x_k}^{2}\right)+\frac{d^2}{2}\norm{y_{k+1}-x_k}^2\\
			\leq&\ \sum_{i=1}^N\sup_{\lambda\geq 0}\frac{N^2L_i^2s_{k+1}^2\lambda^{\alpha_i}-d^2\lambda}{2N}+\frac{d^2}{2}\norm{y_{k+1}-x_k}^2.
		\end{split}
		\label{eq:Fdifferenceupper}
	\end{equation}
	In the case where $\alpha_i=1$, we have
	\[
	N^2L_i^2s_{k+1}^2=N^2L_i^2d^2s^2(k+1)^{-(1-\alpha)}\leq \frac{N^2L_i^2d^2}{N^2L_i^2}\leq d^2,
	\]
	which implies that the supremum vanishes, i.e.,
	\[
	\sup_{\lambda\geq 0}\frac{N^2L_i^2s_{k+1}^2\lambda-d^2\lambda}{2N}=0.
	\]
	If instead $\alpha_i<1$, then a direct maximization gives
	\[
	\sup_{\lambda\geq 0}\frac{N^2L_i^2s_{k+1}^2\lambda^{\alpha_i}-d^2\lambda}{2N}=\frac{(1-\alpha_i)d^2}{2N\alpha_i}\left(\frac{N^2L_i^2s_{k+1}^2\alpha_i}{d^2}\right)^{\frac{1}{1-\alpha_i}}.
	\]
	Recalling the definition of $s_{k+1}$, we obtain
	\begin{align*}
		&\ \frac{(1-\alpha_i)d^2}{2N\alpha_i}\left(\frac{N^2L_i^2s_{k+1}^2\alpha_i}{d^2}\right)^{\frac{1}{1-\alpha_i}}\\
		=&\ \frac{(1-\alpha_i)d^2(N^2L_i^2s^2\alpha_i)^{\frac{1}{1-\alpha_i}}}{2N\alpha_i}\cdot(k+1)^{-\frac{1-\alpha}{1-\alpha_i}}\\
		\leq&\ \frac{(1-\alpha_i)d^2}{2N\alpha_i(k+1)},
	\end{align*}
	where the last inequality follows from the fact that $1-\alpha_i\leq 1-\alpha$. 
	Observe that when $\alpha_i=1$, the right-hand side of the above estimate reduces to $0$. Hence,
	\[
	\sup_{\lambda\geq 0}\frac{N^2L_i^2s_{k+1}^2\lambda^{\alpha_i}-d^2\lambda}{2N}\leq\frac{(1-\alpha_i)d^2}{2N\alpha_i(k+1)}
	\]
	holds uniformly for all $\alpha_i\in(0,1]$.
	
	Furthermore, since $\tilde{g}_{k+1}\in G(y_{k+1})$ and $T$ is monotone, we have the inner product estimate $\inner{F(y_{k+1})+\tilde{g}_{k+1},y_{k+1}-x^*}\geq 0$. 
	Consequently, we derive
	\begin{align*}
		\frac{1}{2}\norm{x_{k+1}-x^*}^2\leq\frac{1}{2}\norm{x_k-x^*}^2-\frac{1-d^2}{2}\norm{y_{k+1}-x_k}^2+\sum_{i=1}^N\frac{(1-\alpha_i)d^2}{2N\alpha_i(k+1)}.
	\end{align*}
	Rearranging the above inequality gives the recursive bound
	\begin{align*}
		\norm{y_{k+1}-x_k}^2\leq\frac{1}{1-d^2}\left(\norm{x_k-x^*}^2-\norm{x_{k+1}-x^*}^2\right)+\sum_{i=1}^N\frac{(1-\alpha_i)d^2}{N(1-d^2)\alpha_i(k+1)}.
	\end{align*}
	
	Next, applying \eqref{eq:Fdifferenceupper} to estimate the term $s_{k+1}^2\norm{F(y_{k+1})+\tilde{g}_{k+1}}^2$, we get
	\begin{align*}
		&\ s_{k+1}^2\norm{F(y_{k+1})+\tilde{g}_{k+1}}^2\\
		\leq&\ 2s_{k+1}^2\norm{F(x_{k})+\tilde{g}_{k+1}}^2+2s_{k+1}^2\norm{F(y_{k+1})-F(x_k)}^2\\
		\leq&\ (2+2d^2)\norm{y_{k+1}-x_k}^2+\sum_{i=1}^N\frac{2(1-\alpha_i)d^2}{N\alpha_i(k+1)}.
	\end{align*}
	Thus, substituting the previous bound yields
	\begin{align*}
		&\ s_{k+1}^2\norm{F(y_{k+1})+\tilde{g}_{k+1}}^2\\
		\leq&\ \frac{2+2d^2}{1-d^2}\left(\norm{x_k-x^*}^2-\norm{x_{k+1}-x^*}^2\right)+\left(2+\frac{2+2d^2}{1-d^2}\right)\sum_{i=1}^N\frac{(1-\alpha_i)d^2}{N\alpha_i(k+1)}\\
		=&\ \frac{2+2d^2}{1-d^2}\left(\norm{x_k-x^*}^2-\norm{x_{k+1}-x^*}^2\right)+\frac{4}{1-d^2}\sum_{i=1}^N\frac{(1-\alpha_i)d^2}{N\alpha_i(k+1)}.
	\end{align*}
	
	Finally, summing the above inequality from $k=0$ to $K$ and invoking the relation $\mathrm{dist}(0,T(y_{k+1}))^2\leq\norm{F(y_{k+1})+\tilde{g}_{k+1}}^2$, we obtain
	\begin{align*}
		&\ \min_{0\leq k\leq K}\mathrm{dist}(0,T(y_{k+1}))^2\sum_{k=0}^K s_{k+1}^2\\
		\leq&\ \min_{0\leq k\leq K}\norm{F(y_{k+1})+\tilde{g}_{k+1}}^2\sum_{k=0}^K s_{k+1}^2\\
		\leq&\ \sum_{k=0}^K s_{k+1}^2\norm{F(y_{k+1})+\tilde{g}_{k+1}}^2\\
		\leq&\ \frac{2+2d^2}{1-d^2}\norm{x_0-x^*}^2+\frac{4}{1-d^2}\sum_{k=0}^K\sum_{i=1}^N\frac{(1-\alpha_i)d^2}{N\alpha_i(k+1)}.
	\end{align*}
	Since $\sum_{k=0}^K s_{k+1}^2$ is of the same order as $\frac{K^\alpha}{\alpha}$, and $\sum_{k=0}^K\frac{1}{k+1}$ is of the same order as $\ln(K)$, it follows that
	\[
	\min_{0\leq k\leq K}\mathrm{dist}(0,T(y_{k+1}))^2\leq O\left(\left(\frac{(2+2d^2)\alpha}{1-d^2}+\sum_{i=1}^N\frac{4(1-\alpha_i)\alpha d^2\ln(K)}{N(1-d^2)\alpha_i}\right)K^{-\alpha}\right).
	\]
\end{proof}
\section{Generalized Accelerated Tseng's Splitting Method}

In this section, we extend two accelerated extensions of the Tseng's splitting method to solve \eqref{eq:main} under the assumption that $F=\sum_{i=1}^NF_i$ and $F_i$ satisfies \eqref{eq:holder} for all $i=1, \cdots, N$, one is the com-FEG method \cite{yang2024}, another one is the SCEG method \cite{yuan2025}. 
\subsection{Generalized Fast Composite Extra-gradient Method}

This section is devoted to study the convergence rate of Algorithm \ref{al:ucfeg}. The main tool for studying Algorithm \ref{al:ucfeg} is the following Lyapunov function:\begin{equation}
	\E(k):=\Sigma_k^2\norm{\tilde{t}_k}^2+\Sigma_k\inner{\tilde{t}_k,x_k-x_0},
	\label{eq:lyapunovh}
\end{equation}
where $\tilde{t}_k=\begin{cases}
	0,& k=0;\\
	F(x_k)+\tilde{g}_k,& k\geq 1.
\end{cases}$ Because $\tilde{g}_{k+1}\in G(x_{k+1})$ for all $k\geq 0$, $\tilde{t}_{k+1}\in T(x_{k+1})$ for all $k\geq 0$.

\begin{algorithm}[!h]
	\caption{Generalized Composite Fast Extra-gradient Method, GCFEG}
	\label{al:ucfeg}
	\begin{algorithmic}[1]
		\State{\textbf{Initialization: } $x_0$, $z_0=x_0$, $\tilde{g}_0=0$, $\Sigma_0=0$, $0<s\leq\min_{1\leq i\leq N}\left(\frac{1}{(N^2L_i^2\alpha_i)^{\frac{1}{2\alpha_i}}}\right)$.}
		\For{$k=0,1,\dots$}
		\State{$s_{k+1}=\frac{s}{(k+1)^{\frac{3(1-\alpha)}{2}}}$.}
		\State{$\beta_{k+1}=\frac{s_{k+1}}{2\Sigma_k+s_{k+1}}$. }
		\State{$\tilde{x}_{k+1}=\beta_{k+1} x_0+(1-\beta_{k+1})x_k$.}
		\State{$x_{k+\frac{1}{2}}=\tilde{x}_{k+1}-(1-\beta_{k+1})s_{k+1}(F(x_k)+\tilde{g}_k)$.}
		\State{$x_{k+1}=J_{s_{k+1}G}\left(\tilde{x}_{k+1}-s_{k+1}F(x_{k+\frac{1}{2}})\right)$.}
		\State{$\tilde{g}_{k+1}=s_{k+1}^{-1}\left(\tilde{x}_{k+1}-x_{k+1}-s_{k+1}F(x_{k+\frac{1}{2}})\right)$.}
		\State{$\Sigma_{k+1}=\Sigma_k+\frac{s_{k+1}}{2}.$}
		\EndFor
	\end{algorithmic}
\end{algorithm}
\begin{remark}
	If $F$ is $L$-Lipschitz continuous and $s_{k+1}\equiv s\in\left(0,\frac{1}{L}\right]$, then $\Sigma_k=\frac{ks}{2}$ and $\beta_{k+1}=\frac{1}{k+1}$. Thus Algorithm \ref{al:ucfeg} is the same as the com-FEG method (with $\rho=0$) in \cite{yang2024}. Yang et al. proved that the convergence rate of Algorithm \ref{al:ucfeg} is\[
	\mathrm{dist}(0,T(x_K))^2\leq O(K^{-2}).
	\]
\end{remark}
\begin{theorem}[Convergence Rate]
	Let $\{x_k\}$, $\{x_{k+\frac{1}{2}}\}$, $\{s_{k+1}\}$ and $\{\Sigma_k\}$ be the sequences generated by Algorithm \ref{al:ucfeg}. If $F=\sum_{i=1}^NF_i$, $F_i$ satisfies \eqref{eq:holder} with $\alpha_i\in\left(\frac{1}{3},1\right]$ for all $i=1, \cdots, N$, $G$ is maximally monotone, $T$ is monotone,  $\alpha=\min_{1\leq i\leq N}\alpha_i$, and $0<s\leq\min_{1\leq i\leq N}\left(\frac{1}{(N^2L_i^2\alpha_i)^{\frac{1}{2\alpha_i}}}\right)$, then we have the following  convergence rate:\[
	\mathrm{dist}(0,T(x_K))^2\leq O\left(\left(\frac{1}{(3\alpha-1)^2s^2}+\sum_{i=1}^N\frac{(1-\alpha_i)\ln(K)}{N\alpha_i}\right)K^{-(3\alpha-1)}\right).
	\]
\end{theorem}
\begin{proof}
	We begin by computing the difference of the sequence $\{\mathcal{E}(k)\}$. 
	A direct calculation gives\begin{align*}
		&\ \E(k+1)-\E(k)\\
		=&\ \Sigma_{k+1}^2\norm{\tilde{t}_{k+1}}^2-\Sigma_k^2\norm{\tilde{t}_k}^2+\Sigma_{k+1}\inner{\tilde{t}_{k+1},x_{k+1}-x_0}+\Sigma_k\inner{\tilde{t}_k,x_k-x_0}\\
		=&\ \Sigma_{k+1}^2\norm{\tilde{t}_{k+1}}^2-\Sigma_k^2\norm{\tilde{t}_k}^2 +\frac{2\Sigma_k+s_{k+1}}{2s_{k+1}}\inner{\tilde{t}_{k+1},(2\Sigma_k+s_{k+1})x_{k+1}-2\Sigma_kx_k-s_{k+1}x_0}\\
		&\ +\frac{\Sigma_k}{s_{k+1}}\inner{\tilde{t}_k,(2\Sigma_k+s_{k+1})x_{k+1}-2\Sigma_kx_k-s_{k+1}x_0}-\frac{(2\Sigma_k+s_{k+1})\Sigma_k}{s_{k+1}}\inner{\tilde{t}_{k+1}-\tilde{t}_k,x_{k+1}-x_k}\\
		\leq&\ \Sigma_{k+1}^2\norm{\tilde{t}_{k+1}}^2-\Sigma_k^2\norm{\tilde{t}_k}^2-\frac{(2\Sigma_k+s_{k+1})^2}{2}\inner{\tilde{t}_{k+1},F(x_{k+\frac{1}{2}})+\tilde{g}_{k+1}}\\
		&\ +(2\Sigma_k+s_{k+1})\Sigma_k\inner{\tilde{t}_k,F(x_{k+\frac{1}{2}})+\tilde{g}_{k+1}}.
	\end{align*}
	First, by using the equation\[
	-\inner{a,b}=\frac{1}{2}\norm{a-b}^2-\frac{1}{2}\norm{a}^2-\frac{1}{2}\norm{b}^2,\quad\forall a, b\in\mathcal{H},
	\]
	we have\begin{align*}
		&\ -\frac{(2\Sigma_k+s_{k+1})^2}{2}\inner{\tilde{t}_{k+1},F(x_{k+\frac{1}{2}})+\tilde{g}_{k+1}}\\
		=&\ \frac{(2\Sigma_k+s_{k+1})^2}{4}\left(\norm{F(x_{k+1})-F(x_{k+\frac{1}{2}})}^2-\norm{\tilde{t}_{k+1}}^2-\norm{F(x_{k+\frac{1}{2}})+\tilde{g}_{k+1}}^2\right).
	\end{align*}
Next, exploiting the decomposition $F=\sum_{i=1}^N F_i$ and the H\"{o}lder condition \eqref{eq:holder} for each $F_i$, we have\begin{align*}
		&\ \frac{(2\Sigma_k+s_{k+1})^2}{4}\norm{F(x_{k+1})-F(x_{k+\frac{1}{2}})}^2\\
		\leq&\ \frac{(2\Sigma_k+s_{k+1})^2}{4}\sum_{i=1}^N NL_i^2\norm{x_{k+1}-x_{k+\frac{1}{2}}}^{2\alpha_i}\\
		=&\ \sum_{i=1}^N\frac{(2\Sigma_k+s_{k+1})^2NL_i^2s_{k+1}^{2\alpha_i}}{4}\norm{F(x_{k+\frac{1}{2}})+\tilde{g}_{k+1}-\frac{2\Sigma_k}{2\Sigma_k+s_{k+1}}\tilde{t}_k}^{2\alpha_i}\\
		&\ -\frac{(2\Sigma_k+s_{k+1})^2}{4}\norm{F(x_{k+\frac{1}{2}})+\tilde{g}_{k+1}-\frac{2\Sigma_k}{2\Sigma_k+s_{k+1}}\tilde{t}_k}^2\\
		&\ +\frac{(2\Sigma_k+s_{k+1})^2}{4}\norm{F(x_{k+\frac{1}{2}})+\tilde{g}_{k+1}-\frac{2\Sigma_k}{2\Sigma_k+s_{k+1}}\tilde{t}_k}^2\\
		\leq&\ \frac{(2\Sigma_k+s_{k+1})^2}{4N}\sum_{i=1}^N\sup_{\lambda\geq 0}(N^2L_i^2s_{k+1}^{2\alpha_i}\lambda^{\alpha_i}-\lambda) -(2\Sigma_k+s_{k+1})\Sigma_k\inner{F(x_{k+\frac{1}{2}})+\tilde{g}_{k+1},\tilde{t}_k}\\
		&\ +\frac{(2\Sigma_k+s_{k+1})^2}{4}\norm{F(x_{k+\frac{1}{2}})+\tilde{g}_{k+1}}^2+\Sigma_k^2\norm{\tilde{t}_k}^2.
	\end{align*}
	If $\alpha_i=1$, then we have \[
	\sup_{\lambda\geq 0}N^2L_i^2s_{k+1}^{2\alpha_i}\lambda-\lambda=\sup_{\lambda\geq 0}(N^2L_i^2s^{2\alpha_i}(k+1)^{-3(1-\alpha)}-1)\lambda=0.
	\]
	If instead $\alpha_i\in\left(\frac{1}{3},1\right)$, then using the relation $s_{k+1}=\frac{s}{(k+1)^{\frac{3(1-\alpha)}{2}}}$ gives
	\begin{align*}
		&\ \sup_{\lambda\geq 0}N^2L_i^2s_{k+1}^{2\alpha_i}\lambda^{\alpha_i}-\lambda\\
		=
		&\ \frac{(1-\alpha_i)(2\Sigma_k+s_{k+1})^2(N^2L_i^2s_{k+1}^{2\alpha_i}\alpha_i)^{\frac{1}{1-\alpha_i}}}{4N\alpha_i}\\
		=&\ \frac{(1-\alpha_i)\Sigma_{k+1}^2(N^2L_i^2s_{k+1}^{2\alpha_i}\alpha_i)^{\frac{1}{1-\alpha_i}}}{N\alpha_i}\\
		=&\ \frac{(1-\alpha_i)\Sigma_{k+1}^2(N^2L_i^2\alpha_i s^{2\alpha_i})^{\frac{1}{1-\alpha_i}}(k+1)^{-\frac{3(1-\alpha)\alpha_i}{1-\alpha_i}}}{N\alpha_i}\\
		\leq&\ \frac{(1-\alpha_i)\Sigma_{k+1}^2}{N\alpha_i(k+1)^{3\alpha}},
	\end{align*}
	where the last inequality follows from $-\frac{3(1-\alpha)\alpha_i}{1-\alpha_i}\leq 3\alpha$. 
	Note that this bound also holds when $\alpha_i=1$. 
	
	 Combining all the preceding estimates, we obtain\begin{align*}
		&\ -\frac{(2\Sigma_k+s_{k+1})^2}{2}\inner{\tilde{t}_{k+1},F(x_{k+\frac{1}{2}})+\tilde{g}_{k+1}}\\
		\leq&\ \sum_{i=1}^N\frac{(1-\alpha_i)(2\Sigma_k+s_{k+1})^2}{4N\alpha_i(k+1)^{3\alpha}}-(2\Sigma_k+s_{k+1})\Sigma_k\inner{F(x_{k+\frac{1}{2}})+\tilde{g}_{k+1},\tilde{t}_k}\\ &\ -\frac{(2\Sigma_k+s_{k+1})^2}{4}\norm{\tilde{t}_{k+1}}^2+\Sigma_k^2\norm{\tilde{t}_k}^2.
	\end{align*}
	Substituting the above inequality back into the expression for $\E(k+1)-\E(k)$ yields the simplified recursive bound\begin{align*}
		 \E(k+1)-\E(k)\leq\sum_{i=1}^N\frac{(1-\alpha_i)(2\Sigma_k+s_{k+1})^2}{4N\alpha_i(k+1)^{3\alpha}}=\sum_{i=1}^N\frac{(1-\alpha_i)\Sigma_{k+1}^2}{N\alpha_i(k+1)^{3\alpha}}.
	\end{align*}
	Moreover, for any $x^*\in\mathrm{SOL}$, we have the following lower bound for $\E(k)$:\begin{align*}
		\E(k)=&\ \Sigma_k^2\norm{\tilde{t}_k}^2+\Sigma_k\inner{\tilde{t}_k,x_k-x^*}+\Sigma_k\inner{\tilde{t}_k,x^*-x_0}\\
		\geq&\ \frac{\Sigma_k^2}{2}\norm{\tilde{t}_k}^2+\Sigma_k\inner{\tilde{t}_k,x_k-x^*}-\frac{1}{2}\norm{x_0-x^*}^2,
	\end{align*}
	which directly implies
	\begin{align*}
		&\ \frac{\Sigma_K^2}{2}\mathrm{dist}(0,T(x_K))^2\leq\frac{\Sigma_K^2}{2}\norm{\tilde{t}_K}^2\\
		\leq&\ \E(K)+\frac{1}{2}\norm{x_0-x^*}^2\\
		=&\ \frac{1}{2}\norm{x_0-x^*}^2+\sum_{k=0}^{K-1}\E(k+1)-\E(k)\\
		\leq&\ \frac{1}{2}\norm{x_0-x^*}^2+\sum_{k=0}^{K-1}\sum_{i=1}^N\frac{(1-\alpha_i)\Sigma_{k+1}^2}{N\alpha_i(k+1)^{3\alpha}}.
	\end{align*}
	Since $\alpha_i\in(\frac{1}{3},1]$ for all $i=1, \cdots, N$, we have $\alpha\in(\frac{1}{3},1]$ and $-\frac{3(1-\alpha)}{2}\in (-1,0)$, which implies that $\Sigma_{k+1}$ is the same order as $\frac{(3\alpha-1)s}{2}k^{1-\frac{3(1-\alpha)}{2}}$. Also, because\[
	2-3(1-\alpha)-3\alpha=-1,
	\]
	 we have\[
	\sum_{k=0}^{K-1}\sum_{i=1}^N\frac{(1-\alpha_i)\Sigma_{k+1}^2}{N\alpha_i(k+1)^{3\alpha}}\sim\sum_{i=1}^N\frac{(1-\alpha_i)(3\alpha-1)^2s^2\ln(K)}{4N\alpha_i}.
	\]
	
	Combining these results, we finally deduce that\begin{align*}
		&\ \mathrm{dist}(0,T(x_K))^2\leq\norm{\tilde{t}_K}^2\\
		\leq&\  \Sigma_K^{-2}\E(0)+\Sigma_K^{-2}\sum_{k=0}^{K-1}\sum_{i=1}^N\frac{(1-\alpha_i)\Sigma_{k+1}^2}{4N\alpha_i(k+1)^{3\alpha}}\\
		\leq&\ O\left(\left(\frac{1}{(3\alpha-1)^2s^2}+\sum_{i=1}^N\frac{(1-\alpha_i)\ln(K)}{N\alpha_i}\right)K^{-(3\alpha-1)}\right).
	\end{align*}
\end{proof}

\subsection{Generalized Symplectic Composite Extra-gradient Method}
This section, we study the convergence rate of Algorithm \ref{al:usceg}.

\begin{algorithm}[h]
	\caption{Generalized Symplectic Composite Extra-gradient Method, GSCEG}
	\label{al:usceg}
	\begin{algorithmic}
		\State{\textbf{Initialization: } $x_0$, $z_0=x_0$, $r>1$, $d\in(0,2(r-1))$, $\Sigma_0=0$, $0<s\leq\min_{1\leq i\leq N}\left(\frac{1}{(N^2L_i^2\alpha_i)^{\frac{1}{2\alpha_i}}}\right)$.}
		\For{$k=0,1,\dots$}
		\State{$s_{k+1}=\frac{s}{(k+1)^{\frac{3(1-\alpha)}{2}}}$.}
		\State{$\beta_{k+1}=\frac{rs_{k+1}}{2\Sigma_k+rs_{k+1}}$. }
		\State{$\tilde{x}_{k+1}=\beta_{k+1} z_k+(1-\beta_{k+1})x_k$.}
		\State{$x_{k+\frac{1}{2}}=\tilde{x}_{k+1}-(1-\beta_{k+1})s_{k+1}(F(x_k)+\tilde{g}_k)$.}
		\State{$x_{k+1}=J_{s_{k+1}G}\left(\tilde{x}_{k+1}-s_{k+1}F(x_{k+\frac{1}{2}})\right)$.}
		\State{$\tilde{g}_{k+1}=s_{k+1}^{-1}\left(\tilde{x}_{k+1}-x_{k+1}-s_{k+1}F(x_{k+\frac{1}{2}})\right)$.}
		\State{$z_{k+1}=z_k-\frac{ds_{k+1}}{2r}(F(x_{k+1})+\tilde{g}_{k+1})$.}
		\State{$\Sigma_{k+1}=\Sigma_k+\frac{s_{k+1}}{2}.$}
		\EndFor
		\Return $x_{k+1}$.
	\end{algorithmic}
\end{algorithm}

\begin{remark}
	If  $F$ is $L$-Lipschitz continuous, then choose $s_{k+1}\equiv s\in(0,\frac{1}{L}]$, we have $\Sigma_k=\frac{ks}{2}$ and $\beta_{k+1}=\frac{r}{k+r}$. Then Algorithm \ref{al:usceg} reduces to the SCEG method (with $\rho=0$). It has been proven that the convergence rate of the SCEG method is $\mathrm{dist}(0,T(x_K))^2\leq O(1/K^2)$ if $s\in(0,\frac{1}{L}]$ and $\mathrm{dist}(0,T(x_K))^2\leq o(1/K^2)$ if $s\in\left(0,\frac{1}{L}\right)$. 
\end{remark}

Similarly, we consider the following Lyapunov function:
\begin{equation}
	\mathcal{E}(k) := d\Sigma_k^2\|\tilde{t}_k\|^2 + dr\Sigma_k\langle \tilde{t}_k, x_k - z_k \rangle + \frac{r^3 - r^2}{2}\norm{z_k-x^*}^2.
	\label{eq:lyapunovconstant}
\end{equation}
where $x^*\in\mathrm{SOL}$, $\tilde{t}_k=\begin{cases}
	0,&k=0;\\
	F(x_k)+\tilde{g}_k,& k\geq 1.
\end{cases}$ By the definition of $J_{s_{k+1}G}$, we have $\tilde{g}_{k+1}\in G(x_{k+1})$, $\tilde{t}_{k+1}\in T(x_{k+1})$, $\forall k\geq 0$.
\begin{lemma}
	\label{lem:lyapunovform}
	The Lyapunov function defined by \eqref{eq:lyapunovconstant} can be transformed into the following form:
	\begin{align*}
		\E(k)=&\ d\Sigma_k^2\left(1-\frac{d}{2(r-1)}\right)\norm{\tilde{t}_k}^2+dr\Sigma_k\inner{\tilde{t}_k,x_k-x^*}\\
		&\ +\frac{1}{2}\norm{\frac{d\Sigma_k}{\sqrt{r-1}}\tilde{t}_k+\sqrt{r^3-r^2}(x^*-z_k)}^2.
	\end{align*}
	If $r>1$, $d\in (0,2(r-1))$, $\E(k)$ can be represented as the sum of three terms.
\end{lemma}
\begin{proof}
	We first observe that
	\[
	dr\Sigma_k\inner{\tilde{t}_k,x_k-z_k}=dr\Sigma_k\inner{\tilde{t}_k,x_k-x^*}+dr\Sigma_k\inner{\tilde{t}_k,x^*-z_k}.
	\]
	Applying the identity
	\[
	\inner{u,v}=\frac{1}{2}\norm{u+v}^2-\frac{1}{2}\norm{u}^2-\frac{1}{2}\norm{v}^2
	\]
	with the choices $u=d\Sigma_k(r-1)^{-\frac{1}{2}}\tilde{t}_k$ and $v=\sqrt{r^3-r^2}(x^*-z_k)$, we obtain
	\begin{align*}
		\mathcal{E}(k)=&\ d\Sigma_k^2\left(1-\frac{d}{2(r-1)}\right)\norm{\tilde{t}_k}^2+dr\Sigma_k\inner{\tilde{t}_k,x_k-x^*}\\&\ +\frac{1}{2}\norm{\frac{d\Sigma_k}{\sqrt{r-1}}\tilde{t}_k+\sqrt{r^3-r^2}(x^*-z_k)}^2.
	\end{align*}
	Since $d\in(0,2(r-1))$, the coefficient $1-\frac{d}{2(r-1)}$ is nonnegative. Moreover, from the monotonicity, we have $\inner{\tilde{t}_k,x_k-x^*}\geq 0$. Therefore, $\mathcal{E}(k)$ is expressed as the sum of three nonnegative terms.
\end{proof}
\begin{lemma}
	\label{lem:nonincreasing}
	Let $\{x_k\}$, $\{x_{k+\frac{1}{2}}\}$, $\{z_k\}$, $\{s_{k+1}\}$ and $\{\Sigma_k\}$ be the sequences generated by Algorithm \ref{al:usceg}. If $F=\sum_{i=1}^N$, $F_i$ satisfies \eqref{eq:holder} with $\alpha_i\in\left(\frac{1}{3},1\right]$ for all $i=1, \cdots, N$, $G$ is maximally monotone, $T$ is monotone, $r>1$,  $d\in (0,2(r-1))$, $\alpha=\min_{1\leq i\leq N}\alpha_i$, and $0<s\leq\min_{1\leq i\leq N}\left(\frac{1}{(N^2L_i^2\alpha_i)^{\frac{1}{2\alpha_i}}}\right)$, then $\{\E(k)\}$ defined by \eqref{eq:lyapunovconstant} satisfies\begin{align*}
		\E(k+1)-\E(k)\leq\sum_{i=1}^N\frac{(1-\alpha_i)(2\Sigma_k+rs_{k+1})^2}{4N\alpha_i(k+1)^{3\alpha}}.
	\end{align*}
\end{lemma}
\begin{proof}
	Step 1: derive some useful equalities and inequalities. We begin with the definition of $\tilde{g}_{k+1}$, which gives\[
	x_{k+1}=\tilde{x}_{k+1}-s_{k+1}(F(x_{k+\frac{1}{2}})+\tilde{g}_{k+1}).
	\] 
	Combining this with the definition of $\tilde{x}_{k+1}$, we obtain the following two useful relations:\begin{align}
		2\Sigma_k(x_{k+1}-x_k) &= rs_{k+1}(z_k-x_{k+1})-(2\Sigma_k+rs_{k+1})s_{k+1}(F(x_{k+\frac{1}{2}})+\tilde{g}_{k+1}),\label{eq:useful1}\\
		rs_{k+1}(x_k-z_k) &= -(2\Sigma_k+rs_{k+1})\left((x_{k+1}-x_k)+s_{k+1}(F(x_{k+\frac{1}{2}})+\tilde{g}_{k+1})\right).\label{eq:useful2}
	\end{align}
	
	Step 2: divide $\E(k+1)-\E(k)$ into three parts. By the definition of $\E(k)$, we have\begin{align*}
		&\ \E(k+1)-\E(k)\\
		 =&\ \underbrace{d\Sigma_{k+1}^2\|\tilde{t}_{k+1}\|^2-d\Sigma_k^2\|\tilde{t}_k\|^2}_{\text{I}} +\underbrace{dr\Sigma_{k+1}\inner{\tilde{t}_{k+1},x_{k+1}-z_{k+1}}-dr\Sigma_k\inner{\tilde{t}_k,x_k-z_k}}_{\text{II}}\\
		&\ +\underbrace{\frac{r^3-r^2}{2}\norm{z_{k+1}-x^*}^2-\frac{r^3-r^2}{2}\norm{z_k-x^*}^2}_{\text{III}}.
	\end{align*}
	
	Step 3: estimate the upper bound of II. We first split II into three sub-terms:\begin{align*}
		\text{II}=&\ \underbrace{\frac{drs_{k+1}}{2}\inner{\tilde{t}_{k+1},x_{k+1}-z_{k+1}}}_{\text{II}_1}+\underbrace{dr\Sigma_k\inner{\tilde{t}_{k+1},x_{k+1}-x_k-z_{k+1}+z_k}}_{\text{II}_2}\\
		&\ +\underbrace{dr\Sigma_k\inner{\tilde{t}_{k+1}-\tilde{t}_k,x_k-z_k}}_{\text{II}_3}.
	\end{align*}
	We now estimate each of these components. From the recursive update of $z_{k+1}$, it follows that\[
	\text{II}_1=\frac{drs_{k+1}}{2}\inner{\tilde{t}_{k+1},x_{k+1}-z_k}+\frac{d^2s_{k+1}^2}{4}\norm{\tilde{t}_{k+1}}^2.
	\]
	Similarly, using \eqref{eq:useful1} and the same recursive relation, we get\begin{align*}
		\text{II}_2=&\ \frac{dr}{2}\inner{\tilde{t}_{k+1},rs_{k+1}(z_k-x_{k+1})-(2\Sigma_k+rs_{k+1})s_{k+1}(F(x_{k+\frac{1}{2}})+\tilde{g}_{k+1})}\\
		&\ +\frac{d^2s_{k+1}\Sigma_k}{2}\norm{\tilde{t}_{k+1}}^2.
	\end{align*}
	Likewise, applying \eqref{eq:useful2} together with the recursive rule of $z_{k+1}$ yields\[
	\text{II}_3=-\frac{d\Sigma_k(2\Sigma_k+rs_{k+1})}{s_{k+1}}\inner{\tilde{t}_{k+1}-\tilde{t}_k,x_{k+1}-x_k+s_{k+1}(F(x_{k+\frac{1}{2}})+\tilde{g}_{k+1})}.
	\]
	Summing the above expressions and simplifying, we obtain the combined form \begin{align*}
		\text{II} =&\ \frac{d(r^2-r)s_{k+1}}{2}\inner{\tilde{t}_{k+1},z_k-x_{k+1}}+\frac{d^2s_{k+1}\Sigma_{k+1}}{2}\norm{\tilde{t}_{k+1}}^2\\
		&\ -\frac{d(2\Sigma_k+rs_{k+1})^2}{2}\inner{\tilde{t}_{k+1},F(x_{k+\frac{1}{2}})+\tilde{g}_{k+1}} +d\Sigma_k(2\Sigma_k+rs_{k+1})\inner{\tilde{t}_k,F(x_{k+\frac{1}{2}})+\tilde{g}_{k+1}}\\
		&\ -\frac{d\Sigma_k(2\Sigma_k+rs_{k+1})}{s_{k+1}}\inner{\tilde{t}_{k+1}-\tilde{t}_k,x_{k+1}-x_k}.
	\end{align*}
	To proceed, we recall the elementary identity\begin{align*}
		&\ -\inner{\tilde{t}_{k+1},F(x_{k+\frac{1}{2}})+\tilde{g}_{k+1}}\\
		=&\ \frac{1}{2}\norm{F(x_{k+1})-F(x_{k+\frac{1}{2}})}^2-\frac{1}{2}\norm{\tilde{t}_{k+1}}^2 -\frac{1}{2}\norm{F(x_{k+\frac{1}{2}})+\tilde{g}_{k+1}}^2.
	\end{align*}
	Using the decomposition $F=\sum_{i=1}^N F_i$ and the H\"{o}lder condition \eqref{eq:holder} for each $F_i$, we estimate \begin{align*}
		&\ \frac{1}{2}\norm{F(x_{k+1})-F(x_{k+\frac{1}{2}})}^2\\
		\leq&\  \sum_{i=1}^N\frac{NL_i^2s_{k+1}^{2\alpha_i}}{2}\norm{\frac{2\Sigma_k}{2\Sigma_k+rs_{k+1}}\tilde{t}_k-(F(x_{k+\frac{1}{2}})+\tilde{g}_{k+1})}^{2\alpha_i}\\
		=&\ \sum_{i=1}^N\frac{NL_i^2s_{k+1}^{2\alpha_i}}{2}\norm{\frac{2\Sigma_k}{2\Sigma_k+rs_{k+1}}\tilde{t}_k-(F(x_{k+\frac{1}{2}})+\tilde{g}_{k+1})}^{2\alpha_i} \\
		&\ -\frac{1}{2}\norm{\frac{2\Sigma_k}{2\Sigma_k+rs_{k+1}}\tilde{t}_k-(F(x_{k+\frac{1}{2}})+\tilde{g}_{k+1})}^2\\
		&\ +\frac{1}{2}\norm{\frac{2\Sigma_k}{2\Sigma_k+rs_{k+1}}\tilde{t}_k-(F(x_{k+\frac{1}{2}})+\tilde{g}_{k+1})}^2\\
		\leq&\ \sum_{i=1}^N\frac{\sup_{\lambda\geq 0}N^2L_i^2s_{k+1}^{2\alpha_i}\lambda^{\alpha_i}-\lambda}{2N} +\frac{1}{2}\norm{\frac{2\Sigma_k}{2\Sigma_k+rs_{k+1}}\tilde{t}_k-(F(x_{k+\frac{1}{2}})+\tilde{g}_{k+1})}^2\\
		\leq&\ \sum_{i=1}^N\frac{1-\alpha_i}{2N\alpha_i(k+1)^{3\alpha}} +\frac{1}{2}\norm{\frac{2\Sigma_k}{2\Sigma_k+rs_{k+1}}\tilde{t}_k-(F(x_{k+\frac{1}{2}})+\tilde{g}_{k+1})}^{2}.
	\end{align*}
	Inserting these estimates back into the expression for II, we arrive at the upper bound\begin{align*}
		\text{II}\leq&\  \frac{d(r^2-r)s_{k+1}}{2}\inner{\tilde{t}_{k+1},z_k-x_{k+1}}+\frac{d^2s_{k+1}\Sigma_{k+1}}{2}\norm{\tilde{t}_{k+1}}^2\\
		&\ +\frac{d(2\Sigma_k+rs_{k+1})^2}{4}\norm{\frac{2\Sigma_k}{2\Sigma_k+rs_{k+1}}\tilde{t}_k-(F(x_{k+\frac{1}{2}})+\tilde{g}_{k+1})}^2\\
		&\ +\frac{d(2\Sigma_k+rs_{k+1})^2}{4}\left(\norm{\tilde{t}_{k+1}}^2+\norm{F(x_{k+\frac{1}{2}})+\tilde{g}_{k+1}}^2\right)+d\Sigma_k(2\Sigma_k+rs_{k+1})\inner{\tilde{t}_k,F(x_{k+\frac{1}{2}})+\tilde{g}_{k+1}}\\
		&\ -\frac{d\Sigma_k(2\Sigma_k+rs_{k+1})}{s_{k+1}}\inner{\tilde{t}_{k+1}-\tilde{t}_k,x_{k+1}-x_k} +\sum_{i=1}^N\frac{(1-\alpha_i)(2\Sigma_k+rs_{k+1})^2}{4N\alpha_i(k+1)^{3\alpha}}\\
		=&\ \frac{d(r^2-r)s_{k+1}}{2}\inner{\tilde{t}_{k+1},z_k-x_{k+1}}+\left(\frac{d^2s_{k+1}\Sigma_{k+1}}{2}-\frac{d(2\Sigma_k+rs_{k+1})^2}{4}\right)\norm{\tilde{t}_{k+1}}^2\\
		&\ +d\Sigma_k^2\norm{\tilde{t}_k}^2-\frac{d\Sigma_k(2\Sigma_k+rs_{k+1})}{s_{k+1}}\inner{\tilde{t}_{k+1}-\tilde{t}_k,x_{k+1}-x_k}+\sum_{i=1}^N\frac{(1-\alpha_i)(2\Sigma_k+rs_{k+1})^2}{4N\alpha_i(k+1)^{3\alpha}}.
	\end{align*}
	
	Step 4: estimate the upper bound of $\E(k+1)-\E(k)$. We turn to part III. Using the identity\[
	\frac{1}{2}\norm{z_{k+1}-x^*}^2-\frac{1}{2}\norm{z_k-x^*}^2=\inner{z_{k+1}-z_k,z_k-x^*}+\frac{1}{2}\norm{z_{k+1}-z_k}^2
	\] 
	and the recursive rule of $z_{k+1}$, we have \[
	\text{III}=-\frac{d(r^2-r)s_{k+1}}{2}\inner{\tilde{t}_{k+1},z_k-x^*}+\frac{d^2(r-1)s_{k+1}^2}{8}\norm{\tilde{t}_{k+1}}^2.
	\]
	Given by the previous calculation about II and III, we have\begin{align*}
		&\ \E(k+1)-\E(k)\\
		\leq & -\left(\frac{d(2\Sigma_k+rs_{k+1})^2}{4}-\frac{d^2s_{k+1}\Sigma_{k+1}}{2}-d\Sigma_{k+1}^2-\frac{d^2(r-1)s_{k+1}^2}{8}\right)\norm{\tilde{t}_{k+1}}^2 \\
		&\ -\frac{d(r^2-r)s_{k+1}}{2}\inner{\tilde{t}_{k+1},x_{k+1}-x^*}-\frac{d\Sigma_k(2\Sigma_k+rs_{k+1})}{s_{k+1}}\inner{\tilde{t}_{k+1}-\tilde{t}_k,x_{k+1}-x_k}\\
		&\ +\sum_{i=1}^N\frac{(1-\alpha_i)(2\Sigma_k+rs_{k+1})^2}{4N\alpha_i(k+1)^{3\alpha}}.
	\end{align*}
	For $k\geq 1$, the monotonicity of $T$ together with $\tilde{t}_{k+1}\in T(x_{k+1})$, $\tilde{t}_k\in T(x_k)$, and $x^*\in\mathrm{SOL}$ implies\begin{align*}
		-\frac{d(r^2-r)s_{k+1}}{2}\inner{\tilde{t}_{k+1},x_{k+1}-x^*}&\leq 0,\quad\forall k\geq 0, \\
		-\frac{d\Sigma_k(2\Sigma_k+rs_{k+1})}{s_{k+1}}\inner{\tilde{t}_{k+1}-\tilde{t}_k,x_{k+1}-x_k} &\leq 0, \quad\forall k\geq 1.
	\end{align*}
	Moreover, since $\Sigma_0=0$, the second inequality also holds for $k=0$. Consequently, both non-positive inner-product terms can be dropped, yielding
	\begin{align*}
		&\ \E(k+1)-\E(k)\\
		\leq&\  -\left(\frac{d(2\Sigma_k+rs_{k+1})^2}{4}-\frac{d^2s_{k+1}\Sigma_{k+1}}{2}-d\Sigma_{k+1}^2-\frac{d^2(r-1)s_{k+1}^2}{8}\right)\norm{\tilde{t}_{k+1}}^2\\
		&\ +\sum_{i=1}^N\frac{(1-\alpha_i)(2\Sigma_k+rs_{k+1})^2}{4N\alpha_i(k+1)^{3\alpha}}.
	\end{align*}
	Using $\Sigma_{k+1}=\Sigma_k+\frac{s_{k+1}}{2}$, we have\begin{align*}
		&\  \frac{d(2\Sigma_k+rs_{k+1})^2}{4}-\frac{d^2s_{k+1}\Sigma_{k+1}}{2}-d\Sigma_{k+1}^2-\frac{d^2(r-1)s_{k+1}^2}{8} \\
		=&\ \frac{d(2\Sigma_{k+1}+(r-1)s_{k+1})^2}{4}-\frac{d^2s_{k+1}\Sigma_{k+1}}{2}-d\Sigma_{k+1}^2-\frac{d^2(r-1)s_{k+1}^2}{8}\\
		=&\ d\left(\Sigma_{k+1}s_{k+1}+\frac{(r-1)s_{k+1}^2}{4}\right)\left(r-1-\frac{d}{2}\right).
	\end{align*}
	Since the coefficient is nonnegative (as $d<2(r-1)$ and the other factors are positive), we obtain the desired recursive bound
	\begin{align*}
		&\ \E(k+1)-\E(k)\\
		\leq &\ -d\left(\Sigma_{k+1}s_{k+1}+\frac{(r-1)s_{k+1}^2}{4}\right)\left(r-1-\frac{d}{2}\right)\norm{\tilde{t}_{k+1}}^2+\sum_{i=1}^N\frac{(1-\alpha_i)(2\Sigma_k+rs_{k+1})^2}{4N\alpha_i(k+1)^{3\alpha}}\\
		\leq&\ \sum_{i=1}^N\frac{(1-\alpha_i)(2\Sigma_k+rs_{k+1})^2}{4N\alpha_i(k+1)^{3\alpha}}.
	\end{align*}
\end{proof}

\begin{theorem}[Convergence Rate]
	Let $\{x_k\}$, $\{x_{k+\frac{1}{2}}\}$, $\{z_k\}$, $\{s_{k+1}\}$ and $\{\Sigma_k\}$ be the sequences generated by Algorithm \ref{al:usceg}. If $F=\sum_{i=1}^NF_i$, where $F_i$ satisfies \eqref{eq:holder} with $\alpha_i\in\left(\frac{1}{3},1\right]$, $G$ is maximally monotone, $T$ is monotone, $r>1$, $d\in (0,2(r-1))$,  $\alpha=\min_{1\leq i\leq N}\alpha_i$, and $0<s\leq\min_{1\leq i\leq N}\left(\frac{1}{(N^2L_i^2\alpha_i)^{\frac{1}{2\alpha_i}}}\right)$, then we have the following last-iterative convergence rate:\[
	\mathrm{dist}(0,T(x_K))^2\leq O\left(\left(\frac{1}{(3\alpha-1)^2s^2}+\sum_{i=1}^N\frac{(1-\alpha_i)\ln(K)}{N\alpha_i}\right)K^{-(3\alpha-1)}\right).
	\]
\end{theorem}
\begin{proof}
	Let $\beta=\frac{3(1-\alpha)}{2}$. Because $\alpha_i\in\left(\frac{1}{3},1\right]$, we have $\beta\in[0,1)$ and $\Sigma_k$ is the same order as $k^{1-\beta}$.
	Next, according to Lemma \ref{lem:nonincreasing}, we have
	\begin{align*}
		&\ \E(K)\\
		\leq&\ \E(0)+\sum_{k=0}^{K-1}\sum_{i=1}^N\frac{(1-\alpha_i)(2\Sigma_k+rs_{k+1})^2}{4N\alpha_i(k+1)^{3\alpha}}\\
		=&\ \E(0)+\sum_{k=0}^{K-1}\sum_{i=1}^N\frac{(1-\alpha_i)(2\Sigma_{k+1}+(r-1)s_{k+1})^2}{4N\alpha_i(k+1)^{3\alpha}}\\
		\leq&\ \E(0)+\sum_{k=0}^{K-1}\sum_{i=1}^N\frac{(1-\alpha_i)(2+(r-1))^2\Sigma_{k+1}^2}{4N\alpha_i(k+1)^{3\alpha}}.
		\end{align*}
 Since $-\frac{3(1-\alpha)}{2}\in (-1,0)$, $\Sigma_{k+1}$ is the same order as $\frac{3\alpha-1}{2}k^{1-\frac{3(1-\alpha)}{2}}$ and \[
		2-3(1-\alpha)-3\alpha=-1,
		\]
		then we have\[
		\sum_{k=0}^{K-1}\sum_{i=1}^N\frac{(1-\alpha_i)\Sigma_{k+1}^2}{4N\alpha_i(k+1)^{3\alpha}}\sim\sum_{i=1}^N\frac{(1-\alpha_i)(3\alpha-1)^2s^2\ln(K)}{4N\alpha_i}.
		\]
		Due to Lemma \ref{lem:lyapunovform}, we have that\begin{align*}
		&\ \mathrm{dist}(0,T(x_K))^2\leq\norm{\tilde{t}_K}^2\\
		\leq&\  \Sigma_K^{-2}\E(0)+\Sigma_K^{-2}\sum_{k=0}^{K-1}\sum_{i=1}^N\frac{(1-\alpha_i)(2+(r-1))^2\Sigma_{k+1}^2}{4N\alpha_i(k+1)^{3\alpha}}\\
		\leq&\ O\left(\left(\frac{1}{(3\alpha-1)^2s^2}+\sum_{i=1}^N\frac{(1-\alpha_i)\ln(K)}{N\alpha_i}\right)K^{-(3\alpha-1)}\right).
		\end{align*}
	
\end{proof}
\section{Numerical Experiments}

\subsection{$l_p$-Norm Minimization over the $l_2$ Ball}

In our first experiment, we consider the following problem:

\begin{equation}
		\min_x \frac{1}{p}\norm{Ax-b}^p_p,\quad	\text{s.t.} \norm{x}_2 \leq 1,
	\label{eq:test1}
\end{equation}
where $p\in(1,2]$, $A\in\mathbb{R}^{m\times n}$, and $b\in\mathbb{R}^m$. Problem \eqref{eq:test1} is a special case of the problem studied in \cite{ito2023}. The gradient of $\frac{1}{p}\norm{Ax-b}^p_p$ is Hölder continuous with Hölder exponent $\alpha = p-1$ and Hölder constant $L = 2^{2-p} m^{\frac{(p-1)(2-p)}{2p}} \norm{A}_2^p$. In our experiments, $m=100$, $n=1000$, the matrix $A$ is randomly generated, and $b$ is generated such that $b = A x^*$ with a randomly generated vector $x^*$ such that $\norm{x^*} \leq 1$. We test four algorithms and compare their numerical performance:
\begin{itemize}
	\item The EG method with $s_{k+1}$ given by \eqref{eq:stepsizeT}, $d=1/\sqrt{2}$, $s=\frac{1}{(L^2\alpha)^{\frac{1}{2\alpha}}}$;
	\item Tseng's splitting method with $d=1/\sqrt{2}$, $s=\frac{1}{(L^2\alpha)^{\frac{1}{2\alpha}}}$;
	\item Algorithm \ref{al:ucfeg} with $s = \frac{1}{(L^2 \alpha)^{\frac{1}{2\alpha}}}$;
	\item Algorithm \ref{al:usceg} with $s = \frac{1}{(L^2 \alpha)^{\frac{1}{2\alpha}}}$, $r=20$, and $d = (r-1)/2$.
\end{itemize}
The numerical results are presented in Figure \ref{fig:1}.

\begin{figure}[!h]
	\centering
	\subfloat{\includegraphics[width=.33\linewidth]{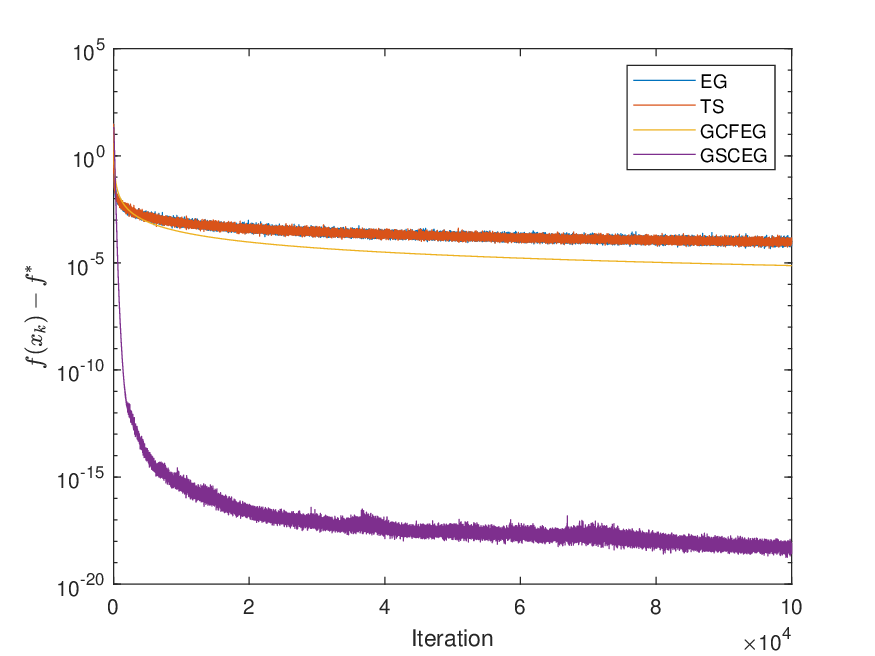}}
	\subfloat{\includegraphics[width=.33\linewidth]{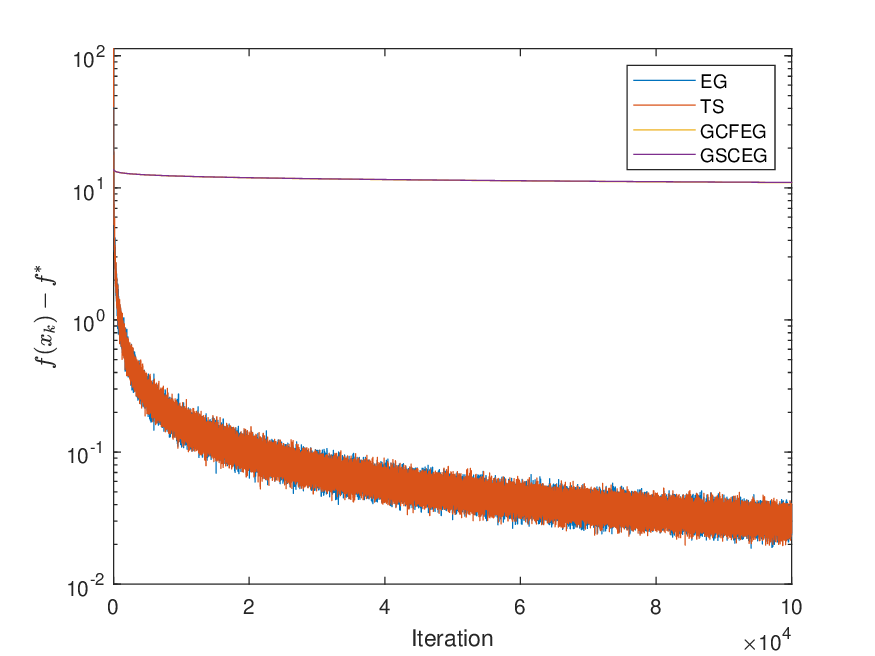}}
	\subfloat{\includegraphics[width=.33\linewidth]{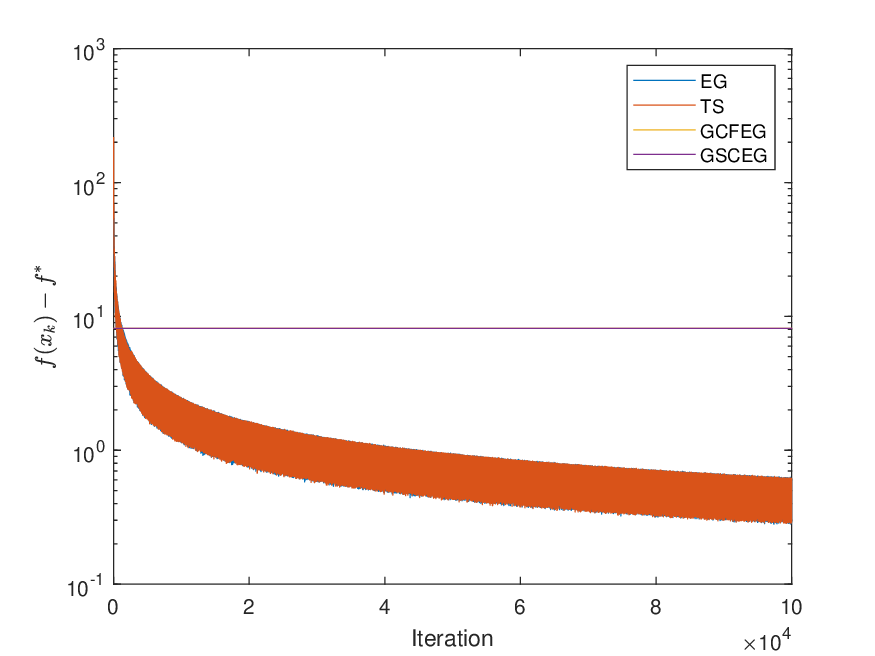}}
	\caption{Numerical results on random instances of problem \eqref{eq:test1} with $m=100$, $n=1000$.}
	\label{fig:1}
\end{figure}

As seen in Figure \ref{fig:1}, when $p=1.8$, Algorithm \ref{al:ucfeg} and Algorithm \ref{al:usceg} converge faster than the EG method and Tseng's splitting method. Moreover, Algorithm \ref{al:usceg} converges the fastest. When $p=1.5$, Algorithm \ref{al:ucfeg} and Algorithm \ref{al:usceg} converge more slowly than the EG method and Tseng's splitting method. When $p=1.2$, Algorithm \ref{al:ucfeg} and Algorithm \ref{al:usceg} do not converge, while the EG method and Tseng's splitting method still converge. It is worth noting, however, that none of the algorithms considered in this paper utilize that $F$ is the gradient of a convex function. Consequently, their numerical performance on problem \eqref{eq:test1} is not as good as that of the parameter-dependent conditional gradient method in \cite{ito2023} (see Figure \ref{fig:1.1}).

\begin{figure}[!h]
	\centering
	\includegraphics[width=.6\linewidth]{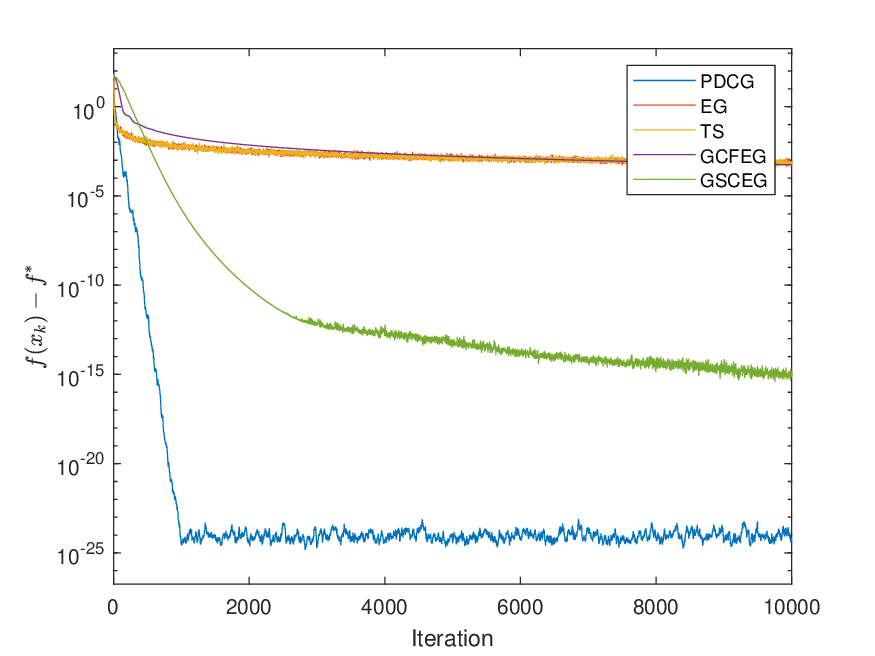}
	\caption{Comparison of the parameter-dependent conditional gradient method with the EG method and its variants.}
	\label{fig:1.1}
\end{figure}

\subsection{Zeros of H\"{o}lder Continuous Operator}
\label{sec:4.2}

In our second numerical experiment, we consider the following nonlinear equation:
\begin{equation}
	0 = F(x), \label{eq:testF}
\end{equation}
where
\begin{equation}
	F(x) = 
	\begin{cases}
		\|x\|^{\alpha-1} (I_{2n} + B)x, & x \neq 0; \\
		0, & x = 0,
	\end{cases}
	\label{eq:F}
\end{equation}
with \(\alpha \in (0,1)\), \(I_{2n} \in \mathbb{R}^{2n \times 2n}\) the \(2n\)-dimensional identity matrix, and \(B \in \mathbb{R}^{2n \times 2n}\) a skew‑symmetric matrix. The operator \(F\) given by \eqref{eq:F} is not the gradient of a function. Moreover, it can be shown that if \(\|B\|_2 \leq \frac{2\sqrt{\alpha}}{1-\alpha}\), then \(F\) is monotone and Hölder continuous with Hölder exponent \(\alpha\) and Hölder constant \(L = 2^{1-\alpha} \sqrt{1 + \|B\|_2^2}\) (see Appendix \ref{sec:A}).

In our experiment, \(n = 1000\) and the matrix \(B\) is chosen as
\[
B = \beta \begin{pmatrix}
	0 & I_n \\
	-I_n & 0
\end{pmatrix},
\]
where \(\beta>0\) is a scalar (its specific value should be provided). It is easy to verify that $\norm{B}_2=\beta$ and \(\mathrm{zero}(F) = \{0\}\). In our test, $\beta$ is set to be $\frac{2\sqrt{\alpha}}{1-\alpha}$, $n=1000$. We test three algorithms and compare their numerical performance under \(\alpha=0.8\), \(\alpha = 0.5\), and \(\alpha = 0.2\), respectively:
\begin{itemize}
	\item The EG method with \(s_{k+1}\) given by \eqref{eq:stepsizeT}, \(d = 1/\sqrt{2}\), and \(s = \frac{1}{(L^2 \alpha)^{\frac{1}{2\alpha}}}\);
	\item Algorithm \ref{al:ucfeg} with \(s = \frac{1}{(L^2 \alpha)^{\frac{1}{2\alpha}}}\);
	\item Algorithm \ref{al:usceg} with \(s = \frac{1}{(L^2 \alpha)^{\frac{1}{2\alpha}}}\), \(r = 20\), and \(d = (r-1)/2\).
\end{itemize}
The Tseng’s splitting method is not included in this test because the EG method coincides with the Tseng’s splitting method when \(G = 0\) (see \cite{tseng00}).

\begin{figure}[!h]
	\centering
	\subfloat{\includegraphics[width=.33\linewidth]{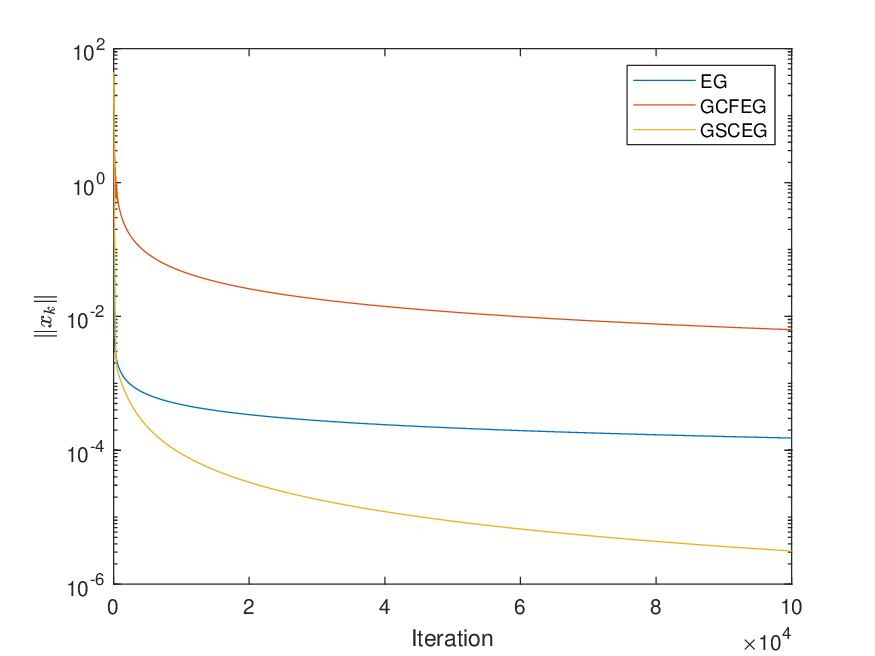}}
	\subfloat{\includegraphics[width=.33\linewidth]{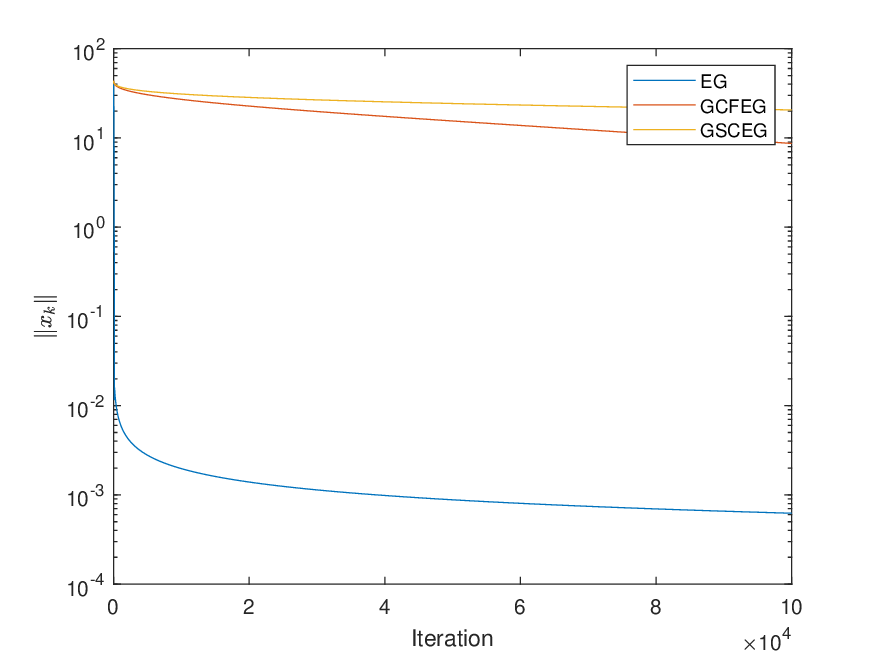}}
	\subfloat{\includegraphics[width=.33\linewidth]{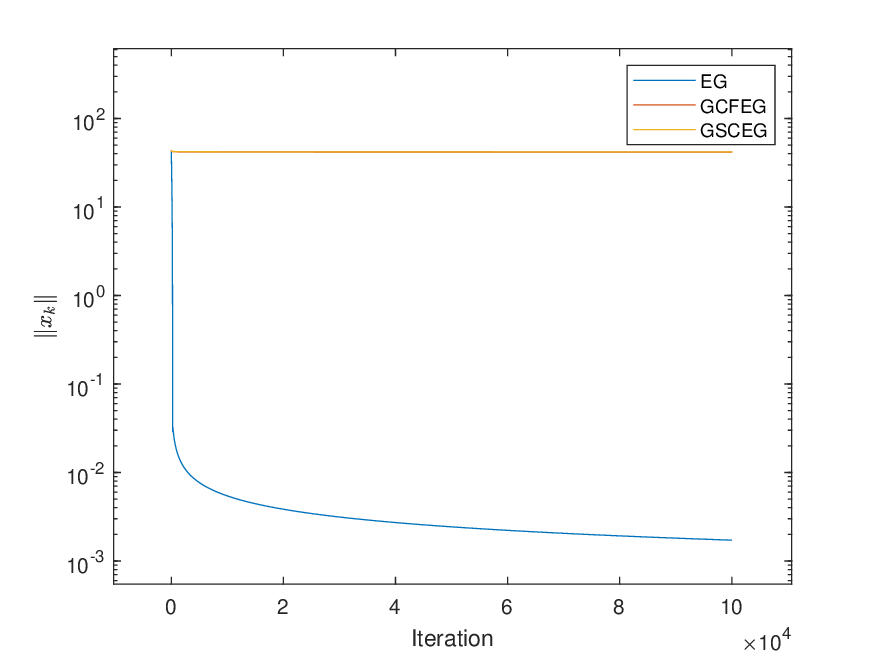}}
	\caption{Numerical results on random instances of problem \eqref{eq:testF} with $n=1000$.}
	\label{fig:2}
\end{figure}

Similar to the previous experiment, when \(\alpha = 0.8\), Algorithm \ref{al:usceg} converges the fastest; when \(\alpha = 0.5\), Algorithm \ref{al:ucfeg} and Algorithm \ref{al:usceg} converge more slowly than the EG method; when \(\alpha = 0.2\), Algorithm \ref{al:ucfeg} and Algorithm \ref{al:usceg} do not converge, while the EG method still converges. It is worth mentioning that when \(\alpha = 0.8\), Algorithm \ref{al:ucfeg} does not converge faster than the EG method, even though its theoretical convergence rate is faster than that of the EG method. A possible explanation is that $\tilde{x}_{k+1}$ in Algorithm \ref{al:ucfeg} is obtain by performing a convex combination between the initial starting point $x_0$ and the current feasible point $x_k$; this may slow down its numerical performance.

\subsection{Minimax Problem with $l_p$ Norm Regularization}
Here, we consider the following minimax problem:

\begin{equation}
	\min_{u \in \mathbb{R}^n} \max_{v \in \mathbb{R}^n} L(u,v) := \frac{1}{2} \langle u, Hu \rangle - \langle h, u \rangle - \langle v, Au - b \rangle + \rho \|u\|_p^p - \rho \|v\|_p^p,
	\label{eq:test2}
\end{equation}
where
\begin{align*}
	A &:= \frac{1}{4} \begin{pmatrix}
		& & & -1 & 1 \\
		& & \iddots & \iddots & \\
		& -1 & 1 & & \\
		-1 & 1 & & & \\
		1 & & & &
	\end{pmatrix} \in \mathbb{R}^{n \times n}, \quad
	H := 2A^\top A, \\
	b &:= \frac{1}{4} \begin{pmatrix} 1 \\ 1 \\ \vdots \\ 1 \\ 1 \end{pmatrix} \in \mathbb{R}^n, \quad
	h := \frac{1}{4} \begin{pmatrix} 0 \\ 0 \\ \vdots \\ 0 \\ 1 \end{pmatrix} \in \mathbb{R}^n.
\end{align*}
When \(\rho = 0\), problem \eqref{eq:test2} reduces to the minimax problem studied in \cite{bot2023fast, ouyang2021, yoon21}. For \(\rho > 0\), problem \eqref{eq:test2} can be reformulated as
\begin{align*}
	F\binom{u}{v} &= F_1\binom{u}{v} + F_2\binom{u}{v}, \quad G = 0, \\
	F_1\binom{u}{v} &= \binom{Hu - h - A^\top v}{Au - b}, \quad
	F_2\binom{u}{v} = \rho \binom{\nabla_u \|u\|_p^p}{\nabla_v \|v\|_p^p},
\end{align*}
where \(F_1\) is 1-Lipschitz continuous, and \(F_2\) is Hölder continuous with Hölder exponent \(\alpha = p-1\) and Hölder constant \(L = 2^{2-p} (2n)^{\frac{(p-1)(2-p)}{2p}}\). In our experiments, we set \(n = 1000\) and \(\rho = 10\). For this problem, the EG method and the generalized Tseng's splitting method share the same iteration formula; therefore, we test the following algorithms:
\begin{itemize}
	\item The EG method with $d=1/\sqrt{2}$ and $s=\min\left\{1, \frac{1}{2L\sqrt{\alpha}}\right\}$;
	\item Algorithm \ref{al:ucfeg} with \(s = \min\left\{\frac{1}{2^2}, \frac{1}{(2^2 L^2 \alpha)^{\frac{1}{2\alpha}}}\right\}\);
	\item Algorithm \ref{al:usceg} with \(s = \min\left\{\frac{1}{2^2}, \frac{1}{(2^2 L^2 \alpha)^{\frac{1}{2\alpha}}}\right\}\), \(r = 20\), and \(d = (r-1)/2\).
\end{itemize}
The numerical results are displayed in Figure \ref{fig:3}.

\begin{figure}[!h]
	\centering
	\subfloat{\includegraphics[width=.33\linewidth]{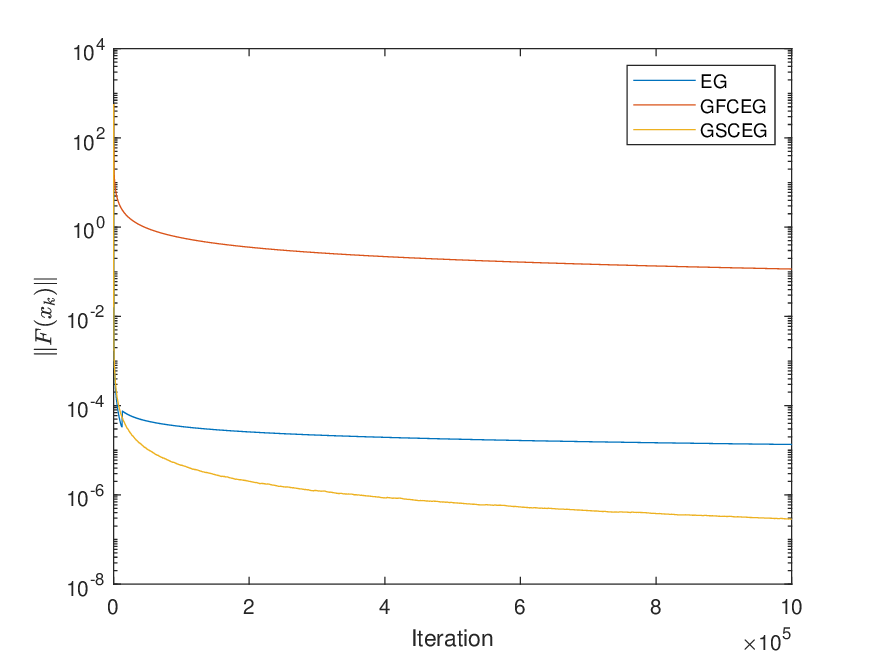}}
	\subfloat{\includegraphics[width=.33\linewidth]{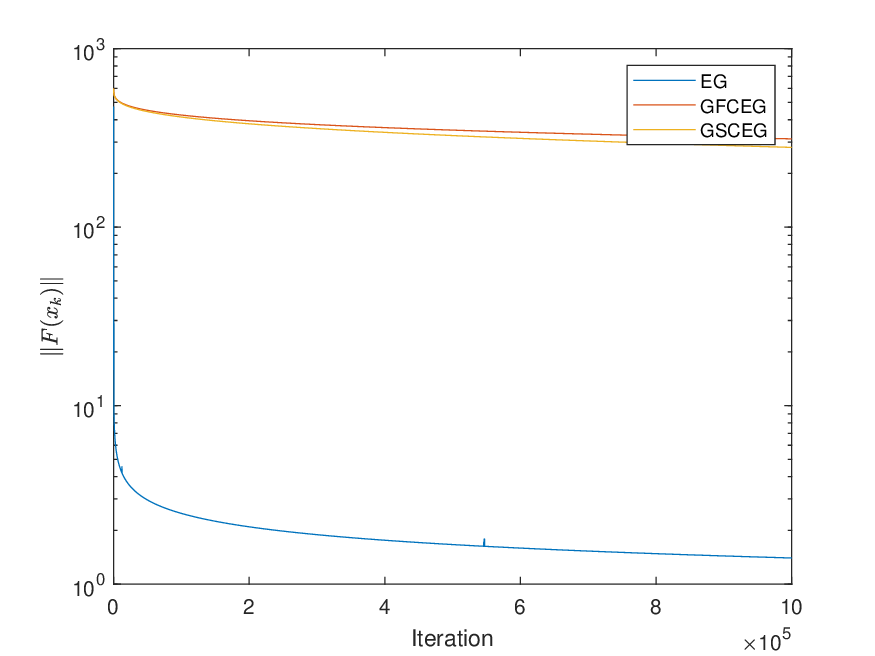}}
	\subfloat{\includegraphics[width=.33\linewidth]{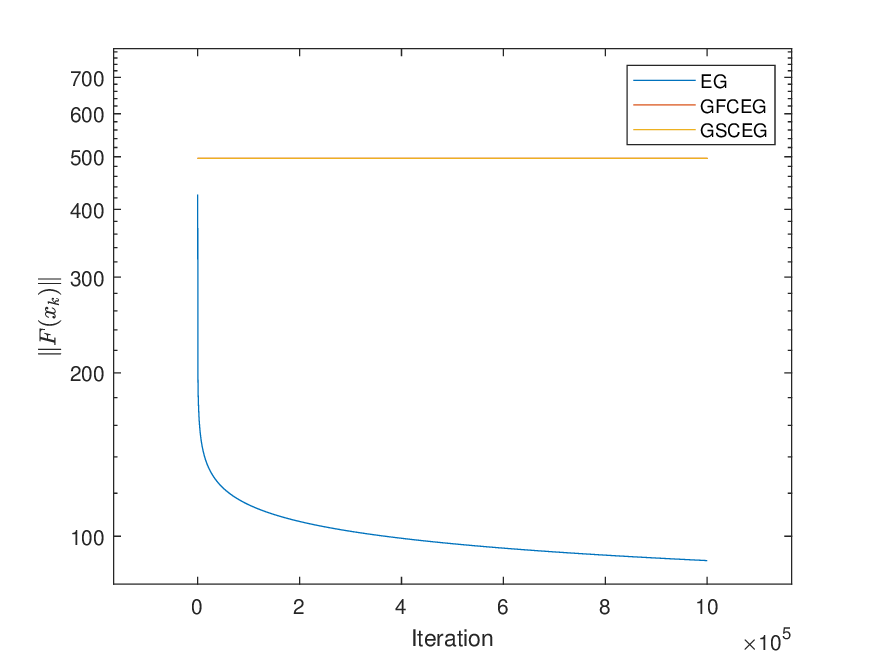}}
	\caption{Numerical results for problem \eqref{eq:test2} with $n=1000$.}
	\label{fig:3}
\end{figure}

As can be seen from Figure \ref{fig:3}, the numerical results are similar to those reported in Section \ref{sec:4.2}.

\section{Conclusion}

In this paper, we consider the monotone inclusion problem \eqref{eq:main} and study the convergence rates of the Tseng's splitting method and its accelerated variants, the com-EAG method and the SCEG method, under the assumption that \(F = \sum_{i=1}^N F_i\), where each \(F_i\) is H\"{o}lder continuous with exponent \(\alpha_i \in (0,1]\) and constant \(L_i\), \(G\) is a maximally monotone operator, and \(T\) is a monotone operator.  
First, we prove that the convergence rate of Tseng's splitting method satisfies  
\[
\min_{0 \leq k \leq K} \operatorname{dist}\bigl(0, T(y_{k+1})\bigr)^2 \leq O\!\left( \left( \frac{(2+2d^2)\alpha}{1-d^2} + \sum_{i=1}^N \frac{(1-\alpha_i)\alpha \ln(K)}{N(1-d^2)\alpha_i} \right) K^{-\alpha} \right).
\]  
We further prove that the convergence rates of Algorithm \ref{al:ucfeg} and Algorithm \ref{al:usceg} satisfy  
\[
\operatorname{dist}\bigl(0, T(x_K)\bigr)^2 \leq O\!\left( \left( \frac{1}{(3\alpha-1)^2 s^2} + \sum_{i=1}^N \frac{(1-\alpha_i)\ln(K)}{N\alpha_i} \right) K^{-(3\alpha-1)} \right).
\]  
When \(\alpha > \frac{1}{2}\), the com‑EAG and SCEG methods converge faster than the Tseng's splitting method; when \(\alpha < \frac{1}{2}\), Tseng's method converges faster than the two accelerated methods. Our numerical experiments also support this theoretical result.

\appendix
\section{Properties of $F$ in Section \ref{sec:4.2}}
\label{sec:A}
We consider the operator $F: \mathbb{R}^n \to \mathbb{R}^n$ as:
 \[F(x) = \begin{cases} \|x\|^{\alpha-1}(I + B)x, & \text{if } x \neq 0; \\ 0, & \text{if } x = 0, \end{cases} \]
where $\alpha\in(0,1)$ and $B \in \mathbb{R}^{n \times n}$ is a skew-symmetric matrix.

\begin{proposition}
	If $\|B\|_2 \leq \frac{2\sqrt{\alpha}}{1-\alpha}$, $F$ is H\"{o}lder continuous with  H\"{o}lder constant $L = 2^{1-\alpha}\sqrt{1 + \|B\|_2^2}$.
\end{proposition}
\begin{proof}
	We seek to show that there exists a constant $L > 0$ such that for all $x, y \in \mathbb{R}^n$,	$ \|F(x) - F(y)\| \leq L\|x - y\|^\alpha.$ Let $A = I + B$. Since $A$ is a linear operator, we can factor it out:
	$$ \|F(x) - F(y)\| = \left\| A \left( x\|x\|^{\alpha-1} - y\|y\|^{\alpha-1} \right) \right\| \leq \|A\|_2 \left\| x\|x\|^{\alpha-1} - y\|y\|^{\alpha-1} \right\| $$
	First, we determine the spectral norm of $A$, given by $\|A\|_2 = \sqrt{\lambda_{\max}(A^T A)}$. Because $B$ is skew-symmetric, we have:
	$$ A^T A = (I + B)^T(I + B) = (I - B)(I + B) = I - B^2 $$
	The eigenvalues of a real skew-symmetric matrix are purely imaginary or zero, meaning the eigenvalues of $B^2$ are real and non-positive. Consequently, the maximum eigenvalue of $I - B^2$ is exactly $1 + \|B\|_2^2$, yielding:
	$$ \|A\|_2 = \sqrt{1 + \|B\|_2^2} $$
	
	Next, we prove that \[
	\left\| x\|x\|^{\alpha-1} - y\|y\|^{\alpha-1} \right\|\leq 2^{1-\alpha}\norm{x-y}^\alpha.
	\]
	Let \(r=\|x\|,\; s=\|y\|\) and let $\theta\in[0,\pi]$ such that \(\cos\theta=\frac{x\cdot y}{rs}\). 
	\[
	\|x\|x\|^{\alpha-1}-y\|y\|^{\alpha-1}\|^2 = r^{2\alpha}+s^{2\alpha}-2r^\alpha s^\alpha\cos\theta = (r^\alpha-s^\alpha)^2+2r^\alpha s^\alpha(1-\cos\theta),
	\]  
	\[
	\|x-y\|^2 = r^2+s^2-2rs\cos\theta = (r-s)^2+2rs(1-\cos\theta).
	\]  
	Let \(t=1-\cos\theta\in[0,2]\), the desired inequality can be transformed into
	\[
	(r^\alpha-s^\alpha)^2+2r^\alpha s^\alpha t \;\le\; 4^{1-\alpha}\bigl[(r-s)^2+2rs t\bigr]^\alpha.
	\]  
	The left-hand side of the above inequality is denoted as \(L(t)\), the right-hand side of the above inequality is denoted as \(R(t)\). Because \(\alpha\in(0,1)\), the function \(t\mapsto R(t)\) is concave, and \(L(t)\) is linear, thus \(H(t)=L(t)-R(t)\) is a convex function. Because the maximum of a convex function attains at the boundary of domain. So we discuss $H(0)$ and $H(2)$. 
\begin{itemize}
	\item \(t=0\): Then we need to prove
	\((r^\alpha-s^\alpha)^2 \le 4^{1-\alpha}(r-s)^{2\alpha}\), i.e. \(|r^\alpha-s^\alpha|\le 2^{1-\alpha}|r-s|^\alpha\).
	Because of the inequality \(|r^\alpha-s^\alpha|\le|r-s|^\alpha\) and  \(2^{1-\alpha}\ge1\), \(|r^\alpha-s^\alpha|\le 2^{1-\alpha}|r-s|^\alpha\) holds.
	\item \(t=2\): Then we need to prove \((r^\alpha+s^\alpha)^2 \le 4^{1-\alpha}(r+s)^{2\alpha}\), i.e.  \(r^\alpha+s^\alpha \le 2^{1-\alpha}(r+s)^\alpha\).	Because \(u\mapsto u^\alpha\) is concave, then using the Jensen's inequality, we have
	\[
	\frac{r^\alpha+s^\alpha}{2}\le \left(\frac{r+s}{2}\right)^\alpha\quad\Longrightarrow\quad r^\alpha+s^\alpha\le 2^{1-\alpha}(r+s)^\alpha.
	\]  
\end{itemize}
	In conclusion, \[
	\left\| x\|x\|^{\alpha-1} - y\|y\|^{\alpha-1} \right\|\leq 2^{1-\alpha}\norm{x-y}^\alpha.
	\]
	
	Based on the previous discussion\[
	\|F(x) - F(y)\| \leq 2^{1-\alpha}\sqrt{1+\norm{B}^2}\norm{x-y}^\alpha.
	\]
	 \end{proof}

\begin{proposition}
	If $\|B\|_2 \leq \frac{2\sqrt{\alpha}}{1-\alpha}$, the operator $F$ is monotone.
\end{proposition}
\begin{proof}
	To prove that $F$ is monotone, we must show that $\langle F(x) - F(y), x - y \rangle \geq 0$ for all $x, y \in \mathbb{R}^n$. Because $F$ is continuous on $\mathbb{R}^n$ and continuously differentiable on $\mathbb{R}^n \setminus \{0\}$, it is sufficient to prove that the symmetric part of its Jacobian matrix, $J_{sym}(x) = \frac{1}{2}(J(x) + J(x)^T)$, is positive semi-definite for all $x \neq 0$.
	
	For $x \neq 0$, the Jacobian matrix $J(x)$ is computed as:
	$$ J(x) = \|x\|^{\alpha-1}(I + B) + (\alpha-1)\|x\|^{\alpha-3}(I + B)xx^T $$
		Taking the transpose and using the property $B^T = -B$, we find:
	$$ J(x)^T = \|x\|^{\alpha-1}(I - B) + (\alpha-1)\|x\|^{\alpha-3}xx^T(I - B) $$
	Averaging $J(x)$ and $J(x)^T$ yields the symmetric part:
	$$ J_{sym}(x) = \|x\|^{\alpha-1}I + (\alpha-1)\|x\|^{\alpha-3}xx^T + \frac{\alpha-1}{2}\|x\|^{\alpha-3} \left( Bxx^T - xx^TB \right) $$
	
	We must demonstrate that for any non-zero vector $v \in \mathbb{R}^n$, the quadratic form $v^T J_{sym}(x) v \geq 0$. Because \[
	 v^T (Bxx^T - xx^TB) v = (v^TBx)(x^Tv) - (v^Tx)(x^TBv),
	\]	
	and $x^TBv = (B^Tx)^T v = -(Bx)^Tv = -v^TBx$, we have $v^T (Bxx^T - xx^TB) v=2(v^Tx)(v^TBx)$. Then we have	
	$$ v^T J_{sym}(x) v = \|x\|^{\alpha-1} \left[ \|v\|^2 - (1-\alpha)\frac{(v^Tx)^2}{\|x\|^2} - (1-\alpha)\frac{(v^Tx)(v^TBx)}{\|x\|^2} \right] $$
	
	Let $\hat{x} = \frac{x}{\|x\|}$ be the unit vector in the direction of $x$. We can decompose $v$ into a component parallel to $\hat{x}$ and a component orthogonal to it. Without loss of generality, assume $\|v\| = 1$ and write $v = \cos\theta \hat{x} + \sin\theta \hat{w}$, where $\hat{w}$ is a unit vector orthogonal to $\hat{x}$ ($\hat{w}^T\hat{x} = 0$).
	
	First, we evaluate the two inner product term.  $v^T\hat{x} = \cos\theta$; $v^TB\hat{x} = (\cos\theta \hat{x}^T + \sin\theta \hat{w}^T)B\hat{x} = \sin\theta (\hat{w}^TB\hat{x})$, exploiting the fact that $\hat{x}^TB\hat{x} = 0$. Substituting these into the bracketed term of the quadratic form yields:
	$$ \Gamma := 1 - (1-\alpha)\cos^2\theta - (1-\alpha)\cos\theta\sin\theta (\hat{w}^TB\hat{x}) $$
	Using the identity $1 = \sin^2\theta + \cos^2\theta$, this becomes:
	$$ \Gamma = \sin^2\theta + \alpha\cos^2\theta - (1-\alpha)(\hat{w}^TB\hat{x})\cos\theta\sin\theta $$
	
	For $J_{sym}(x)$ to be positive semi-definite, we require $\Gamma \geq 0$ for all $\theta$ and all allowable $\hat{w}$. We can view $\Gamma$ as a quadratic form in terms of variables $X = \cos\theta$ and $Y = \sin\theta$:
	$$ Q(X, Y) = \alpha X^2 + Y^2 - (1-\alpha)(\hat{w}^TB\hat{x})XY $$
	This quadratic form is non-negative everywhere if and only if its discriminant with respect to the polynomial in $X/Y$ is less than or equal to zero:
	$$ \left( (1-\alpha)(\hat{w}^TB\hat{x}) \right)^2 - 4\alpha\leq 0, $$
	which leads to
	$$ |\hat{w}^TB\hat{x}| \leq \frac{2\sqrt{\alpha}}{1-\alpha}. $$
	
	By the definition of the spectral norm and Cauchy-Schwarz inequality, $|\hat{w}^TB\hat{x}| \leq \|\hat{w}\| \|B\hat{x}\| \leq \|B\|_2$.
	Therefore, if we enforce the bound $\|B\|_2 \leq \frac{2\sqrt{\alpha}}{1-\alpha}$, the discriminant condition is universally satisfied. Consequently, $J_{sym}(x)$ is positive semi-definite for all $x \neq 0$, proving that $F$ is monotone on $\mathbb{R}^n$. 
\end{proof}

\section*{Acknowledgments}
The author thank the anonymous reviewers for their valuable suggestions. The author also wants to thank Prof. Xiaojun Chen for providing guidance to this paper. This work is supported in part by funds from RGC grant JLFS/P-501/24 for the CAS AMSS-PolyU Joint Laboratory in Applied Mathematics and Hong Kong Research Grant Council project PolyU15300024. The author uses DeepSeek to improve readability and to assist in writing the code.


\bibliographystyle{abbrvnat}
\bibliography{reference}

@article{dang2015,
  title={On the convergence properties of non-euclidean extragradient methods for variational inequalities with generalized monotone operators},
  author={Dang, Cong D and Lan, Guanghui},
  journal={Computational Optimization and applications},
  volume={60},
  number={2},
  pages={277--310},
  year={2015},
  publisher={Springer},
  doi={10.1007/s10589-014-9673-9}
}

@article{Nemirovski2004,
author = {Nemirovski, Arkadi},
title = {Prox-Method with Rate of Convergence O(1/t) for Variational Inequalities with Lipschitz Continuous Monotone Operators and Smooth Convex-Concave Saddle Point Problems},
journal = {SIAM Journal on Optimization},
volume = {15},
number = {1},
pages = {229-251},
year = {2004},
doi = {10.1137/S1052623403425629}
}

@article{stonyakin2022,
  title={Generalized mirror prox algorithm for monotone variational inequalities: Universality and inexact oracle},
  author={Stonyakin, Fedor and Gasnikov, Alexander and Dvurechensky, Pavel and Titov, Alexander and Alkousa, Mohammad},
  journal={Journal of Optimization Theory and Applications},
  volume={194},
  number={3},
  pages={988--1013},
  year={2022},
  publisher={Springer},
  doi={10.1007/s10957-022-02062-7}
}

@misc{chakraborty2025,
      title={Popov Mirror-Prox Method for Variational Inequalities}, 
      author={Abhishek Chakraborty and Angelia Nedić},
      year={2025},
      eprint={2507.23395},
      archivePrefix={arXiv},
      primaryClass={math.OC},
      url={https://arxiv.org/abs/2507.23395}, 
}

@article{korpelevich76,
	title={The extragradient method for finding saddle points and other problems},
	author={Korpelevich, Galina M},
	journal={Matecon},
	volume={12},
	pages={747--756},
	year={1976}
}

@article{antipin1976,
	title={On a method for convex programs using a symmetrical modification of the {L}agrange function},
	author={Anatoly S. Antipin},
	year={1976},
	journal={Economika i Matem. Metody},
	volumn={7},
	pages={1164-1173},
	number={6}
}

@article{popov80,
	title={A modification of the {A}rrow-{H}urwitz method of search for saddle points},
	author={Leonid D. Povov},
	journal={Mat. Zametki},
	volume={28},
	number={5},
	pages={777--784},
	year={1980}
}

@InProceedings{daskalakis2018a,
	author = {Daskalakis, Constantinos and Panageas, Ioannis},
	title = {The limit points of (optimistic) gradient descent in min-max optimization},
	year = {2018},
	booktitle = {Proceedings of the 32nd International Conference on Neural Information Processing Systems},
	pages = {9256–9266},
	numpages = {11},
	location = {Montr\'{e}al, Canada},
	series = {NIPS'18},
	url={https://dl.acm.org/doi/pdf/10.5555/3327546.3327597}
}

@InProceedings{daskalakis2018b,
	title={Training {GAN}s with Optimism},
	author={Constantinos Daskalakis and Andrew Ilyas and Vasilis Syrgkanis and Haoyang Zeng},
	booktitle={International Conference on Learning Representations},
	year={2018},
	url={https://openreview.net/forum?id=SJJySbbAZ}
}

@article{tseng00,
	title={A modified forward-backward splitting method for maximal monotone mappings},
	author={Paul Tseng},
	journal={SIAM Journal on Control and Optimization},
	volume={38},
	number={2},
	pages={431--446},
	year={2000},
	publisher={SIAM},
	doi={10.1137/S0363012998338806}
}

@article{cai2021,
  title={Strong convergence theorems for solving variational inequality problems with pseudo-monotone and non-Lipschitz operators},
  author={Cai, Gang and Dong, Qiao-Li and Peng, Yu},
  journal={Journal of Optimization Theory and Applications},
  volume={188},
  number={2},
  pages={447--472},
  year={2021},
  publisher={Springer},
  doi={10.1007/s10957-020-01792-w}
}

@inproceedings{diakonikolas21,
	title={Efficient methods for structured nonconvex-nonconcave min-max optimization},
	author={Jelena Diakonikolas and Constantinos Daskalakis  and Michael I. Jordan},
	booktitle={International Conference on Artificial Intelligence and Statistics},
	pages={2746--2754},
	year={2021},
	organization={PMLR}
}

@inproceedings{fan23,
	title={Weaker {MVI} Condition: Extragradient Methods with Multi-Step Exploration},
	author={Yifeng Fan and Yongqiang Li  and  Bo Chen },
	booktitle={The Twelfth International Conference on Learning Representations},
	year={2023}
}

@inproceedings{pethick22,
	title={Escaping limit cycles: Global convergence for constrained nonconvex-nonconcave minimax problems},
	author={Thomas Pethick and Puya Latafat and Panos Patrinos and Olivier Fercoq and Volkan Cevher},
	booktitle={International Conference on Learning Representations},
	year={2022},
	url={https://openreview.net/forum?id=2\_vhkAMARk}
}

@inproceedings{gorbunov22,
	title={Stochastic extragradient: General analysis and improved rates},
	author={Eduard Gorbunov and  Hugo Berard and Gauthier Gidel and Nicolas Loizou},
	booktitle={International Conference on Artificial Intelligence and Statistics},
	pages={7865--7901},
	year={2022},
	organization={PMLR}
}

@article{iusem17,
	title={Extragradient method with variance reduction for stochastic variational inequalities},
	author={Alfredo N. Iusem and Alekandro Jofr\'{e} and Roberto  I. Oliveira and Philip Thompson},
	journal={SIAM Journal on Optimization},
	volume={27},
	number={2},
	pages={686--724},
	year={2017},
	publisher={SIAM},
	doi={10.1137/15M1031953}
}

@article{kannan19,
	title={Optimal stochastic extragradient schemes for pseudomonotone stochastic variational inequality problems and their variants},
	author={Kannan, Aswin and Shanbhag, Uday V},
	journal={Computational Optimization and Applications},
	volume={74},
	number={3},
	pages={779--820},
	year={2019},
	publisher={Springer},
	doi={10.1007/s10589-019-00120-x}
}

@inproceedings{mishchenko20,
	title={Revisiting stochastic extragradient},
	author={Konstantin Mishchenko and Dmitry Kovalev and  Egor Shulgin and Peter Richt{\'a}rik and Yura Malitsky },
	booktitle={International Conference on Artificial Intelligence and Statistics},
	pages={4573--4582},
	year={2020},
	organization={PMLR}
}

@misc{tran23,
	author = {Quoc Tran-Dinh},
	title = {Sublinear Convergence Rates of Extragradient-Type Methods: A Survey on Classical and Recent Developments},
	year = {2023},
	url={https://arxiv.org/abs/2303.17192}, 
}

@article{halpern67,
	title={Fixed points of nonexpanding maps},
	author={Benjamin Halpern},
	journal={Bulletin of the American Mathematical Society},
	year={1967},
	volume={73},
	pages={957-961},
	url={https://api.semanticscholar.org/CorpusID:120539954}
}

@article{lieder2021,
  title={On the convergence rate of the {H}alpern-iteration},
  author={Lieder, Felix},
  journal={Optimization letters},
  volume={15},
  number={2},
  pages={405--418},
  year={2021},
  publisher={Springer},
  doi = {10.1007/s11590-020-01617-9},
}

@article{qi21,
	title={Convergence of {H}alpern’s iteration method with applications in optimization},
	author={Qi, Huiqiang and Xu, Hong-Kun},
	journal={Numerical Functional Analysis and Optimization},
	volume={42},
	number={15},
	pages={1839--1854},
	year={2021},
	publisher={Taylor \& Francis},
	doi={https://doi.org/10.1080/01630563.2021.2001826}
}

@article{sun2025a,
author = {Sun, Defeng and Yuan, Yancheng and Zhang, Guojun and Zhao, Xinyuan},
title = {Accelerating Preconditioned {ADMM} via Degenerate Proximal Point Mappings},
journal = {SIAM Journal on Optimization},
volume = {35},
number = {2},
pages = {1165-1193},
year = {2025},
doi = {10.1137/24M1650053}
}

@article{chen2025,
  title={{HPR-LP}: An implementation of an {HPR} method for solving linear programming},
  author={Chen, Kaihuang and Sun, Defeng and Yuan, Yancheng and Zhang, Guojun and Zhao, Xinyuan},
  journal={Mathematical Programming Computation},
  pages={1--28},
  year={2025},
  publisher={Springer},
  doi = {10.1007/s12532-025-00292-0},
}

@misc{lu2024,
      title={Restarted {H}alpern {PDHG} for Linear Programming}, 
      author={Haihao Lu and Jinwen Yang},
      year={2024},
      eprint={2407.16144},
      archivePrefix={arXiv},
      primaryClass={math.OC},
      url={https://arxiv.org/abs/2407.16144}
}

@inproceedings{yoon21,
	title={Accelerated Algorithms for Smooth Convex-Concave Minimax Problems with ${O}(1/k^2)$ Rate on Squared Gradient Norm},
	author={TaeHo Yoon and Ernest K. Ryu},
	booktitle={International Conference on Machine Learning},
	pages={12098--12109},
	year={2021},
	organization={PMLR}
}

@inproceedings{lee21,
	title={Fast Extra Gradient Methods for Smooth Structured Nonconvex-Nonconcave Minimax Problems},
	author={Sucheol Lee and Donghwan Kim},
	booktitle={Advances in Neural Information Processing Systems},
	editor={A. Beygelzimer and Y. Dauphin and P. Liang and J. Wortman Vaughan},
	year={2021}
}

@InProceedings{yang2024,
	title = 	 {Accelerated Algorithms for Constrained Nonconvex-Nonconcave Min-Max Optimization and Comonotone Inclusion},
	author =       {Cai, Yang and Oikonomou, Argyris and Zheng, Weiqiang},
	booktitle = 	 {Proceedings of the 41st International Conference on Machine Learning},
	pages = 	 {5312--5347},
	year = 	 {2024},
	editor = 	 {Salakhutdinov, Ruslan and Kolter, Zico and Heller, Katherine and Weller, Adrian and Oliver, Nuria and Scarlett, Jonathan and Berkenkamp, Felix},
	volume = 	 {235},
	series = 	 {Proceedings of Machine Learning Research},
	month = 	 {21--27 Jul},
	url = 	 {https://proceedings.mlr.press/v235/cai24f.html},
}

@article{yuan2026,
  title={Symplectic discretization approach for developing new proximal point algorithm},
  author={Yuan, Ya-xiang and Zhang, Yi},
  journal={Computational Optimization and Applications},
  volume={93},
  number={2},
  pages={651--687},
  year={2026},
  publisher={Springer},
  doi = {10.1007/s10589-025-00728-2},
}

@misc{yuan2025,
      title={Symplectic Extra-gradient Type Method for Solving General Non-monotone Inclusion Problem}, 
      author={Ya-xiang Yuan and Yi Zhang},
      year={2025},
      eprint={2406.10793},
      archivePrefix={arXiv},
      primaryClass={math.OC},
      url={https://arxiv.org/abs/2406.10793}, 
}

@misc{trandinh2025,
      title={Accelerated Extragradient-Type Methods -- Part 2: Generalization and Sublinear Convergence Rates under Co-Hypomonotonicity}, 
      author={Quoc Tran-Dinh and Nghia Nguyen-Trung},
      year={2025},
      eprint={2501.04585},
      archivePrefix={arXiv},
      primaryClass={math.OC},
      url={https://arxiv.org/abs/2501.04585}, 
}

@article{tran2024,
	title={Extragradient-type methods with $\mathcal{O}(1/k)$ last-iterate convergence rates for co-hypomonotone inclusions},
	author={Quoc Tran-Dinh},
	journal={Journal of Global Optimization},
	volume={89},
	number={1},
	pages={197--221},
	year={2024},
	publisher={Springer},
	doi={10.1007/s10898-023-01347-z}
}

@article{bot2023fast,
	title={Fast Optimistic Gradient Descent Ascent ({OGDA}) method in continuous and discrete time},
	author={Bo{\c{t}}, Radu Ioan and Csetnek, Ern{\"o} Robert and Nguyen, Dang-Khoa},
	journal={Foundations of Computational Mathematics},
	pages={1--60},
	year={2023},
	publisher={Springer},
	doi={https://doi.org/10.1007/s10208-023-09636-5}
}

@inproceedings{bot2023project,
	author = {Sedlmayer, Michael and Nguyen, Dang-Khoa and Bo\c{t}, Radu Ioan},
	title = {A fast optimistic method for monotone variational inequalities},
	year = {2023},
	booktitle = {Proceedings of the 40th International Conference on Machine Learning},
	articleno = {1261},
	numpages = {33},
	location = {Honolulu, Hawaii, USA},
	series = {ICML'23}
}

@article{qu2026extra,
  title={An extra gradient {A}nderson-accelerated algorithm for pseudomonotone variational inequalities},
  author={Qu, Xin and Bian, Wei and Chen, Xiaojun},
  journal={Mathematics of Computation},
  volume={95},
  number={359},
  pages={1361--1387},
  year={2026},
  doi = {10.1090/mcom/4095},
}

@article{anderson1965,
  title={Iterative procedures for nonlinear integral equations},
  author={Anderson, Donald G},
  journal={Journal of the ACM (JACM)},
  volume={12},
  number={4},
  pages={547--560},
  year={1965},
  publisher={ACM New York, NY, USA},
  doi = {10.1145/321296.321305},
}

@article{ito2023,
  author  = {Masaru Ito and Zhaosong Lu and Chuan He},
  title   = {A Parameter-Free Conditional Gradient Method for Composite Minimization under {H}\"{o}lder Condition},
  journal = {Journal of Machine Learning Research},
  year    = {2023},
  volume  = {24},
  number  = {166},
  pages   = {1--34},
  url     = {http://jmlr.org/papers/v24/22-0983.html}
}

@article{ouyang2021,
  title={Lower complexity bounds of first-order methods for convex-concave bilinear saddle-point problems},
  author={Ouyang, Yuyuan and Xu, Yangyang},
  journal={Mathematical Programming},
  volume={185},
  number={1},
  pages={1--35},
  year={2021},
  publisher={Springer},
  doi={10.1007/s10107-019-01420-0}
}

@article{TOYASAKI2014340,
  title={A variational inequality formulation of equilibrium models for end-of-life products with nonlinear constraints},
  author={Toyasaki, Fuminori and Daniele, Patrizia and Wakolbinger, Tina},
  journal={European Journal of Operational Research},
  volume={236},
  number={1},
  pages={340--350},
  year={2014},
  publisher={Elsevier},
  doi = {10.1016/j.ejor.2013.12.006}
}

@book{kinderlehrer2000,
  title={An introduction to variational inequalities and their applications},
  author={Kinderlehrer, David and Stampacchia, Guido},
  year={2000},
  publisher={SIAM}
}

@article{UI2016139,
title = {Bayesian {N}ash equilibrium and variational inequalities},
journal = {Journal of Mathematical Economics},
volume = {63},
pages = {139-146},
year = {2016},
issn = {0304-4068},
doi = {10.1016/j.jmateco.2016.02.004},
author = {Takashi Ui}
}

@article{taylor84,
	title={An interpretation for min-max structural design problems including a method for relaxing constraints},
	author={J. E. Taylor and Martin P. Bends{\o}e },
	journal={International Journal of Solids and Structures},
	volume={20},
	number={4},
	pages={301--314},
	year={1984},
	publisher={Elsevier},
	doi={10.1016/0020-7683(84)90041-6}
}

@book{bental09,
	title={Robust optimization},
	author={Aharon Ben-Tal and Laurent El Ghaoui and Arkadi Nemirovski},
	volume={28},
	year={2009},
	publisher={Princeton university press},
	address={Princeton}
}

@article{minty62,
	author = {George J. Minty},
	title = {{Monotone (nonlinear) operators in Hilbert space}},
	volume = {29},
	journal = {Duke Mathematical Journal},
	number = {3},
	publisher = {Duke University Press},
	pages = {341 -- 346},
	year = {1962},
	doi = {10.1215/S0012-7094-62-02933-2},
	URL = {https://doi.org/10.1215/S0012-7094-62-02933-2}
}

@article{nesterov83,
	title={A method of solving a convex programming problem with convergence rate $O(1/k^2)$},
	author={Yurii Nesterov},
	journal={Doklady Akademii Nauk SSSR},
	volume={269},
	number={3},
	pages={543},
	year={1983}
}

\end{document}